\documentclass[10pt]{IEEEtran}

\usepackage{amsmath,amsfonts,amsthm,amssymb,authblk,array,pxfonts,paralist}
\usepackage[bookmarks=false, colorlinks=true, linkcolor=blue!50!red, citecolor=orange,
pdfencoding=unicode]{hyperref}
\usepackage[capitalize]{cleveref}
\usepackage[svgnames]{xcolor}
\usepackage{tikz}

\usetikzlibrary{
	cd,
	shapes.geometric,
	decorations.markings,
	decorations.pathmorphing,
	positioning,
	arrows,
	shapes,
	calc,
	fit,
	quotes}
\usepackage{todonotes}

\newcommand{\lrn}{\mathsf{Learn}}

\newcommand{\set}{\mathsf{Set}}
\newcommand{\fvect}{\mathsf{FVect}}
\newcommand{\nnet}{\mathsf{NNet}}
\newcommand{\para}{\mathsf{Para}}
\newcommand{\idag}{\mathsf{IDAG}}

\newcommand{\rr}{\mathbb{R}}
\newcommand{\nn}{\mathbb{N}}
\newcommand{\ve}{\varepsilon}
\newcommand{\prl}{\!\parallel\!}
\newcommand{\define}[1]{{\textbf{#1}}}

\newcommand{\id}{\mathrm{id}}
\newcommand{\norm}[1]{\lVert #1\rVert}

\newtheorem*{theorem*}{Theorem}
\newtheorem{definition}{Definition}[section]
\newtheorem{proposition}[definition]{Proposition}   
\newtheorem{theorem}[definition]{Theorem}

\theoremstyle{remark}
\newtheorem{example}[definition]{Example}

\newtheorem*{example*}{Example}
\newtheorem{remark}[definition]{Remark}

\pgfdeclarelayer{edgelayer}
\pgfdeclarelayer{nodelayer}
\pgfsetlayers{edgelayer,nodelayer,main}

\tikzstyle{none}=[inner sep=0pt]
\tikzstyle{ibox}=[draw, rounded corners, minimum width = 30pt, minimum height =
18pt, thick]
\tikzstyle{update}=[->,>=stealth, very thick,decoration={snake, pre length =3pt, post
length =3pt},decorate]

\tikzset{
   oriented WD/.style={
      every to/.style={out=0,in=180,draw},
      label/.style={
         font=\everymath\expandafter{\the\everymath\scriptstyle},
         inner sep=0pt,
         node distance=2pt and -2pt},
      semithick,
      node distance=1 and 1,
      decoration={markings, mark=at position .5 with {\arrow{stealth};}},
      ar/.style={postaction={decorate}},
      execute at begin picture={\tikzset{
         x=\bbx, y=\bby,
         every fit/.style={inner xsep=\bbx, inner ysep=\bby}}}
      },
   bbx/.store in=\bbx,
   bbx = 1.5cm,
   bby/.store in=\bby,
   bby = 1.75ex,
   bb port sep/.store in=\bbportsep,
   bb port sep=2,
   bb port length/.store in=\bbportlen,
   bb port length=0pt,
   bb penetrate/.store in=\bbpenetrate,
   bb penetrate=0pt,
   bb min width/.store in=\bbminwidth,
   bb min width=1cm,
   bb rounded corners/.store in=\bbcorners,
   bb rounded corners=5pt,
   bb small/.style={bb port sep=1, bb port length=2.5pt, bbx=.4cm, bb min width=.4cm, bby=.7ex},
   bb/.code 2 args={
      \pgfmathsetlengthmacro{\bbheight}{\bbportsep * (max(#1,#2)) * \bby}
      \pgfkeysalso{draw,minimum height=\bbheight,minimum width=\bbminwidth,outer sep=0pt,
         rounded corners=\bbcorners,thick,
         prefix after command={\pgfextra{\let\fixname\tikzlastnode}},
         append after command={\pgfextra{\draw
            \ifnum #1=0{} \else foreach \i in {1,...,#1} {
               ($($(\fixname.north
	       west)+(0,.9\bbportsep)$)!{\i/(#1+1)}!($(\fixname.south
	       west)-(0,.9\bbportsep)$)$)
	       +(-\bbportlen,0) coordinate (\fixname_in\i) -- +(\bbpenetrate,0) coordinate (\fixname_in\i')}\fi 
            \ifnum #2=0{} \else foreach \i in {1,...,#2} {
               ($($(\fixname.north
	       east)+(0,\bbportsep)$)!{\i/(#2+1)}!($(\fixname.south
	       east)-(0,\bbportsep)$)$) +(-\bbpenetrate,0) coordinate (\fixname_out\i') -- +(\bbportlen,0) coordinate (\fixname_out\i)}\fi;
         }}}
   },
   bb name/.style={append after command={\pgfextra{\node[anchor=north] at
(\fixname.north) {#1};}}},
   ibb port sep/.store in=\ibbportsep,
   ibb port sep=2,
   ibb port length/.store in=\ibbportlen,
   ibb port length=4pt,
   ibb min width/.store in=\ibbminwidth,
   ibb min width=1cm,
   ibb rounded corners/.store in=\ibbcorners,
   ibb rounded corners=1pt,
   ibb small/.style={ibb port sep=1, ibb port length=2.5pt, bbx=.4cm, ibb min width=.4cm, bby=.7ex},
   ibb/.code 2 args={
	   \pgfmathsetlengthmacro{\ibbheight}{\ibbportsep * (max(#1,#2)) * \bby}
	   \pgfkeysalso{draw,color=gray!50,minimum height=\ibbheight,minimum width=\ibbminwidth,outer sep=0pt,
		   rounded corners=\ibbcorners,thick,
		   prefix after command={\pgfextra{\let\fixname\tikzlastnode}},
		   append after command={\pgfextra{\coordinate
			   \ifnum #1=0{} \else foreach \i in {1,...,#1} {
				   ($($(\fixname.north
					west)+(0,.9\ibbportsep)$)!{\i/(#1+1)}!($(\fixname.south
						west)-(0,.9\ibbportsep)$)$)
					   +(-\ibbportlen,0) coordinate (\fixname_in\i) -- +(\ibbportlen,0) coordinate (\fixname_in\i')}\fi 
					   \ifnum #2=0{} \else foreach \i in {1,...,#2} {
						   ($($(\fixname.north
							east)+(0,\ibbportsep)$)!{\i/(#2+1)}!($(\fixname.south
								east)-(0,\ibbportsep)$)$) +(-\ibbportlen,0) coordinate (\fixname_out\i') -- +(\ibbportlen,0) coordinate (\fixname_out\i)}\fi;
		   }}}
   },
   ibb name/.style={append after command={\pgfextra{\node[anchor=north] at
   (\fixname.north) {#1};}}},
   blankbb port sep/.store in=\blankbbportsep,
   blankbb port sep=2,
   blankbb min width/.store in=\blankbbminwidth,
   blankbb min width=1cm,
   blankbb rounded corners/.store in=\blankbbcorners,
   blankbb rounded corners=1pt,
   blankbb small/.style={blankbb port sep=1, blankbb port length=2.5pt, bbx=.4cm, blankbb min width=.4cm, bby=.7ex},
   blankbb/.code 2 args={
	   \pgfmathsetlengthmacro{\blankbbheight}{\blankbbportsep * (max(#1,#2)) * \bby}
	   \pgfkeysalso{draw,color=gray!50,minimum height=\blankbbheight,minimum width=\blankbbminwidth,outer sep=0pt,
		   rounded corners=\blankbbcorners,thick,
		   prefix after command={\pgfextra{\let\fixname\tikzlastnode}},
		   append after command={\pgfextra{\draw
			\ifnum #1=0{} \else foreach \i in {1,...,#1} {
			   ($($(\fixname.north
			   west)+(0,.9\ibbportsep)$)!{\i/(#1+1)}!($(\fixname.south
			   west)-(0,.9\ibbportsep)$)$)
					    coordinate (\fixname_in\i)}\fi 
			\ifnum #2=0{} \else foreach \i in {1,...,#2} {
			   ($($(\fixname.north
			   east)+(0,.9\ibbportsep)$)!{\i/(#2+1)}!($(\fixname.south
			   east)-(0,.9\ibbportsep)$)$) coordinate (\fixname_out\i)}\fi;
		   }}}
   },
   blankbb name/.style={append after command={\pgfextra{\node[anchor=north] at
     (\fixname.north) {#1};}}},
   symbb port sep/.store in=\symbbportsep,
   symbb port sep=2,
   symbb port length/.store in=\symbbportlen,
   symbb port length=0pt,
   symbb min width/.store in=\symbbminwidth,
   symbb min width=1cm,
   symbb rounded corners/.store in=\symbbcorners,
   symbb rounded corners=5pt,
   symbb small/.style={symbb port sep=1, symbb port length=2.5pt, symbbx=.4cm, symbb min width=.4cm, symbby=.7ex},
   symbb/.code 2 args={
      \pgfmathsetlengthmacro{\symbbheight}{\symbbportsep * (max(#1,#2)) * \bby}
      \pgfkeysalso{draw,minimum height=\symbbheight,minimum width=\symbbminwidth,outer sep=0pt,
         rounded corners=\symbbcorners,thick,
         prefix after command={\pgfextra{\let\fixname\tikzlastnode}},
         append after command={\pgfextra{\draw
            \ifnum #1=0{} \else foreach \i in {1,...,#1} {
               ($($(\fixname.north
	       west)+(0,.9\symbbportsep)$)!{\i/(#1+1)}!($(\fixname.south
	       west)-(0,.9\symbbportsep)$)$)
	       +(-\symbbportlen,0) coordinate (\fixname_in\i) -- +(\symbbportlen,0) coordinate (\fixname_in\i')}\fi 
            \ifnum #2=0{} \else foreach \i in {1,...,#2} {
               ($($(\fixname.north
	       east)+(0,.9\symbbportsep)$)!{\i/(#2+1)}!($(\fixname.south
	       east)-(0,.9\symbbportsep)$)$) +(-\symbbportlen,0) coordinate (\fixname_out\i') -- +(\symbbportlen,0) coordinate (\fixname_out\i)}\fi;
         }}}
   },
   symbb name/.style={append after command={\pgfextra{\node[anchor=north] at
(\fixname.north) {#1};}}},
}

\title{Backprop as Functor:\\A compositional perspective on supervised learning}
\author{
  \makebox[.45\linewidth]{Brendan Fong \hspace{1cm} David Spivak} 
  \hspace{.5cm}
  \makebox[.45\linewidth]{R\'emy Tuy\'eras} 
\\
\makebox[.45\linewidth]{Department of Mathematics,}
  \hspace{.5cm}
\makebox[.45\linewidth]{Computer Science and Artificial Intelligence
Lab,}
\\ 
\makebox[.45\linewidth]{Massachusetts Institute of Technology }
  \hspace{.5cm}
\makebox[.45\linewidth]{Massachusetts Institute of Technology }
\thanks{We thank Patrick
Schultz and Amalie Trewartha for useful discussions. Work supported by AFOSR
FA9550-14-1-0031 and FA9550-17-1-0058.}
}
\date{}

  \IEEEoverridecommandlockouts
  \IEEEpubid{\makebox[\columnwidth]{978-1-7281-3608-0/19/\$31.00~
  \copyright2019 IEEE \hfill} \hspace{\columnsep}\makebox[\columnwidth]{ }}

\begin{document}
\maketitle

\begin{abstract}
A supervised learning algorithm searches over a set of functions $A \to B$
parametrised by a space $P$ to find the best approximation to some ideal
function $f\colon A \to B$. It does this by taking examples $(a,f(a)) \in
A\times B$, and updating the parameter according to some rule. We define a
category where these update rules may be composed, and show that gradient
descent---with respect to a fixed step size and an error function satisfying a
certain property---defines a monoidal functor from a category of parametrised
functions to this category of update rules. A key contribution is the notion of
request function. This provides a structural perspective on backpropagation,
giving a broad generalisation of neural networks and linking it with structures
from bidirectional programming and open games.
\end{abstract}

\section{Introduction}\label{sec.intro}

Machine learning, and in particular the use of neural networks, has rapidly
become remarkably effective at real world tasks \cite{Nie}. A significant
contributor to this success has been the backpropagation algorithm.
Backpropagation gives a way to compute the derivative of a function via message
passing on a network, significantly speeding up learning. Yet, while the power
of this approach has been impressive, it is also somewhat mysterious. What
structures make backpropagation so effective, and how can we interpret, predict,
and generalise it?

In recent years, monoidal categories have been used to formalise the use of
networks in computation and reasoning---amongst others, applications include
circuit diagrams, Markov processes, quantum computation, and dynamical systems
\cite{Fon,BFP,CK,VSL}. This paper responds to a need for
more structural approaches to machine learning by using categories to provide an
algebraic, compositional perspective on learning algorithms and backpropagation.

Consider a supervised learning algorithm. The goal of a supervised learning
algorithm is to find a suitable approximation to a function $f\colon A \to B$.
To do so, the supervisor provides a list of pairs $(a,b) \in A \times B$, each of 
which is supposed to approximate the values taken by $f$, i.e.\ $b \approx f(a)$.
The supervisor also defines a space of functions over which the learning algorithm
will search. This is formalised by choosing a set $P$ and a function
$I\colon P \times A \to B$. We denote the function at parameter $p \in P$ as
$I(p,-)\colon A \to B$. Then, given a pair $(a,b) \in A \times B$, the learning
algorithm takes a current hypothetical approximation of $f$, say given by
$I(p,-)$, and tries to improve it, returning some new best guess, $I(p',-)$.
In other words, a supervised learning algorithm includes an \emph{update} function $U\colon
P \times A \times B \to P$ for $I$.

\begin{center}
\begin{tikzpicture}
  \node[ibox] (i1) at (0,2.5) {$I(p,-)$};
  \node (a1) at (-2.5,2.5) {$a$};
  \node (b1) at (2.5,2.5) {$b$};
  \node[ibox] (i2) at (0,0) {$I(p',-)$};
  \node (a2) at (-2.5,0) {};
  \node (b2) at (2.5,0) {};
  \draw (a1) -- (i1) -- (b1);
  \draw (a2) -- (i2) -- (b2);
  \draw[update] ($(i1.south)-(0,.2)$) to node[left] {\small \texttt{update}}
  node[right] {\small $(a,b)$} ($(i2.north)+(0,.2)$) ;
  \node[below=.5cm, align=justify, text width=8cm] (cap) at (i2) {\small Figure
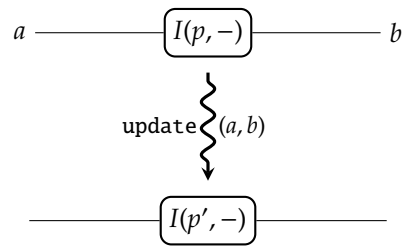
  1. Given a training datum $(a,b)$, a learning algorithm updates $p$ to $p'$.};
\end{tikzpicture}
\end{center}

To make this compositional, we ask the following question. 
Suppose we are given two learning algorithms, as described above, one for
approximating functions $A \to B$ and the other for functions $B \to C$. Can we piece them together to make a learning algorithm for approximating
functions $A \to C$? We will see that the answer is no, because something is missing.

To construct a learning algorithm for the composite, we would need a parameterised function $A\to
C$ as well as an update rule. It is easy to take the given parameterised
functions $I\colon P \times A \to B$ and $J\colon Q \times B \to C$ and produce
one from $A$ to $C$. Indeed, take $P\times Q$ as the parameter space and define
the parametrised function $P\times Q\times A\to C;$ $(p,q,a)\mapsto
J(q,I(p,a))$. We call the function $J(-,I(-,-))\colon P \times Q \times A \to C$
the \emph{composite parametrised function}.

The problem comes in defining the update rule for the composite learner.
Our algorithm must take as training data pairs $(a,c)$ in $A\times C$. However, to use
the given update functions, written $U$ and $V$ for updating $I$ and $J$
respectively, we must produce training data
of the form $(a',b')$ in $A\times B$ and $(b'',c'')$ in $B \times C$. It is
straightforward to produce a pair in $B \times C$---take $\big(I(p,a),c\big)$---but there is no natural pair $(a',b')$ to use as training data for $I$. The choice of
$b'$ should encode something about the information in both $c$ and $J$, and nothing of the sort has been specified.

Thus to complete the compositional picture, we must add to our formalism a way
for the second learning algorithm to pass back elements of $B$. We will call
this a \emph{request} function, because it is as though $J$ is telling $I$ what
input $b'$ would have been more helpful. The request function for $J$ will be of
the form $s\colon Q \times B \times C \to B$: given a hypothesis $q$ and
training data $(b'',c'')$, it returns $b'\coloneqq s(q,b'',c'')$. Now we have
the desired training data $(a,b')$ for $I$. The request function is thus a way
of `backpropagating' the output back toward the earlier learners in a network.

\begin{center}
\begin{tikzpicture}
   [dist/.store in=\dist, dist = 5.8,
    dangle/.store in=\dangle, dangle = 1.9,
    font=\small, scale=.8
   ]
  \node[ibox] (i9) at (0,6) {$I(p,-)$};
  \node (a9) at ($(i9)-(\dangle,0)$) {$a$};
  \node[ibox] (j9) at (\dist,6) {$J(q,-)$};
  \node (c9) at ($(j9)+(\dangle,0)$) {$c$};
  \draw (a9) -- (i9) -- (j9) -- (c9);
  \node[ibox] (i0) at (0,4.5) {$I(p,-)$};
  \node (a0) at ($(i0)-(\dangle,0)$) {$a$};
  \node (b0) at ($(i0)+(\dangle,0)$) {};
  \node[ibox] (j0) at (\dist,4.5) {$J(q,-)$};
  \node (c0) at ($(j0)+(\dangle,0)$) {$c$};
  \node (b0') at ($(j0)-(\dangle,0)$) {$I(p,a)$};
  \draw (a0) -- (i0) -- (b0);
  \draw (b0')-- (j0) -- (c0);
  \draw[very thick, bend left=15, >=stealth,->, color=gray]
  (i0.north east) to node[above] {\texttt{implement}} (b0'.north);
  \node[ibox] (i1) at (0,3) {$I(p,-)$};
  \node (a1) at ($(i1)-(\dangle,0)$) {$a$};
  \node (b1) at ($(i1)+(\dangle,0)$) {$s(q,I(p,a),c)$};
  \node[ibox] (j1) at (\dist,3) {$J(q,-)$};
  \node (c1) at ($(j1)+(\dangle,0)$) {$c$};
  \node (b1') at ($(j1)-(\dangle,0)$) {$I(p,a)$};
  \draw (a1) -- (i1) -- (b1);
  \draw (b1')-- (j1) -- (c1);
  \draw[very thick, bend right=15, >=stealth,->, color=gray]
  (j1.north west) to node[above] {\texttt{request}} (b1.north);
  \node[ibox] (i2) at (0,0.5) {$I(p',-)$};
  \node (a2) at ($(i2)-(\dangle,0)$) {};
  \node[ibox] (j2) at (\dist,0.5) {$J(q',-)$};
  \node (c2) at ($(j2)+(\dangle,0)$) {};
  \draw (a2) -- (i2) -- (j2) -- (c2);
  \draw[update] ($(i1.south)-(0,.2)$) to node[left] {\small \texttt{update}}
  node[right] {\small $(a,s(q,I(p,a),c))$} ($(i2.north)+(0,.2)$) ;
  \draw[update] ($(j1.south)-(0,.2)$) to node[left] {\small \texttt{update}}
  node[right] {\small $(I(p,a),c)$} ($(j2.north)+(0,.2)$) ;
  \node[below=.5cm, align=justify, text width=8cm] (cap) at
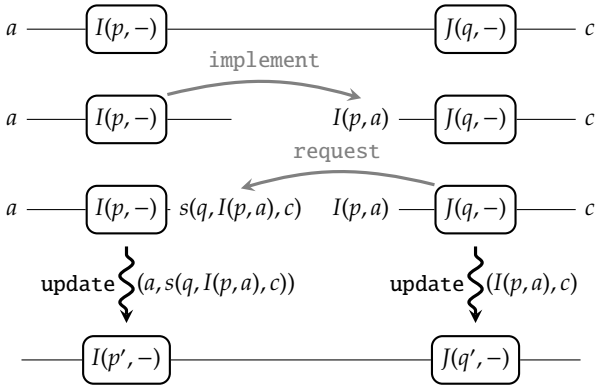
  ($0.5*(i2)+0.5*(j2)$) {\small Figure 2. A request function allows an update
  function to be defined for the composite $J(q,I(p,-))$.};
\end{tikzpicture}
\end{center}

In this paper we will show that learning becomes \emph{compositional}---i.e.\ we can define
a learning algorithm $A \to C$ from learning algorithms $A \to B$ and $B \to C$---as long as each learning algorithm consists of these four components:
\begin{compactitem}
	\item a parameter space $P$,
	\item an implementation function $I\colon P\times A\to B$,
	\item an update function $U\colon P\times A\times B\to P$, and
	\item a request function $r\colon P\times A\times B\to A$.
\end{compactitem}

More precisely, we will show that learning
algorithms $(P,I,U,r)$ form the morphisms of a category $\lrn$. A category is
an algebraic structure that models composition. More precisely, a
category consists of \emph{types} $A$, $B$, $C$, and so on, \emph{morphisms}
$f\colon A \to B$ between these types, and a \emph{composition rule} by which
morphisms $f\colon A \to B$ and $g\colon B \to C$ can be combined to create a morphism
$A\to C$. Thus we can say that learning algorithms form a category, as we have
informally explained above.
In fact, they have more structure because they can be composed not only in series but
also in parallel, and this too has a clean algebraic description. Namely, we will
show that $\lrn$ has the structure of a symmetric monoidal category.

This novel category $\lrn$, synthesised from the above analysis of learning
algorithms, nonetheless curiously resembles and is closely related to lenses
\cite{BPV} and open games \cite{GHWZ}, two well-known structures that also model
compositional, bidirectional exchange information between interacting
systems. We return to this briefly in \cref{sec.discuss}.

Our aim thus far has been to construct an algebraic description of learning
algorithms, and we claim that the category $\lrn$ suffices. In particular, then,
our framework should be broad enough to capture known methods for constructing
supervised learning algorithms; such learning algorithms should sit inside
$\lrn$ as a particular kind of morphism. Here we study neural networks.

Let us say that a \emph{neural network layer of type $(n_1,n_2)$} is a subset $C
\subseteq [n_1]\times [n_2]$, where $n_1,n_2\in\nn$ are natural numbers, and
$[n] = \{1,\dots,n\}$ for any $n\in\nn$. The numbers $n_1$ and $n_2$ represent
the number of nodes on each side of the layer, $C$ is the set of connections,
and the inclusion $C\subseteq[n_1]\times[n_2]$ encodes the connectivity
information, i.e.\ $(i,j)\in C$ means node $i$ on the right is connected to node $j$ on the left.

If we additionally fix a function $\sigma\colon \rr \to \rr$, which we call
the \emph{activation function}, then a neural network layer defines a
parametrised function $I\colon\rr^{\lvert C\rvert+n_2}\times
\rr^{n_1}\to\rr^{n_2}$. The $\rr^{\lvert C \rvert}$ factor encodes numbers
called \emph{weights} and the $\rr^{n_2}$ factor encodes numbers called the
\emph{biases}. For example, the layer $C=\{(1,1),(2,1),(2,2)\}\subseteq [2]
\times [2]$, has $n_2=2$ biases and $\lvert C\rvert=3$ weights. The biases
are represented by the right hand nodes below, while the weights are
represented by the edges:
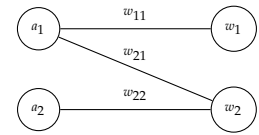
\[
\begin{tikzpicture}[font=\tiny]
\node[draw, circle] (i1) {$a_2$};
\node[draw, circle, above=0.5 of i1] (i2) {$a_1$};
\node[draw, circle, right=2 of i1] (o1) {$w_2$};
\node[draw, circle, right=2 of i2] (o2) {$w_1$};
\draw (i1) to[above] node {$w_{22}$} (o1);
\draw (i2) to[above] node {$w_{21}$} (o1);
\draw (i2) to[above] node {$w_{11}$} (o2);
\end{tikzpicture}
\]
This layer defines the parametrised function $I\colon \rr^5
\times \rr^2 \to \rr^2$, given by
\begin{multline*}
I\big(w_{11},w_{21},w_{22},w_1,w_2,a_1,a_2\big) \\ \coloneqq 
\big(\,\sigma(w_{11}a_1+w_1),\; \sigma(w_{21}a_1+w_{22}a_2+w_2)\,\big).
\end{multline*}

A neural network is a sequence of layers of types $(n_0,n_1)$,
$(n_1,n_2)$, $\dots$, $(n_{k-1},n_k)$. By composing the parametrised functions
defined by each layer as above, a neural network itself defines a parametrised
function $P\times\rr^{n_0}\to\rr^{n_k}$ for some $P$. Note that this function is always
differentiable if $\sigma$ is.

To go from a differentiable parametrised function to a learning
algorithm, one typically specifies a suitable error function $e$ and a step size $\ve$,
and then uses an algorithm known as gradient descent. 

Our main theorem is that, under general conditions, gradient descent is
compositional. This is formalised as a functor $\para \to \lrn$,
where $\para$ is a category where morphisms are differentiable
parametrised functions $I\colon P \times \rr^n \to \rr^m$ between finite
dimensional Euclidean spaces, where the parameter space $P=\rr^p$ is also Euclidean. 

In brief, the functoriality means that given two differentiable parametrised
functions $I$ and $J$, we get the same result if we (i) use gradient descent to
get learning algorithms for $I$ and $J$, and then compose those learning
algorithms, or (ii) compose $I$ and $J$ as parametrised functions, and then use
gradient descent to get a learning algorithm. More precisely, we have the
following:

\begin{theorem*}
Fix $\ve>0$ and $e(x,y)\colon \rr \times \rr \to \rr$
such that $\frac{\partial e}{\partial x}(x_0,-)\colon \rr \to \rr$ is invertible
for each $x_0 \in \rr$. Then there is a faithful, injective-on-objects, symmetric monoidal
functor
\[
  L_{\ve,e}\colon\para \longrightarrow \lrn
\]
sending each differentiable parametrised function $I\colon P \times \rr^n \to
\rr^m$ to the learning algorithm $(P,I,U_I,r_I)$ defined by
\[
  U_I(p,a,b) \coloneqq p - \varepsilon\nabla_pE_I(p,a,b)
\]
and
\[
  r_I(p,a,b) \coloneqq f_a\Big(\nabla_a E_I(p,a,b)\Big),
\]
where $E_I(p,a,b) \coloneqq \sum_i e(I(p,a)_i,b_i)$ and $f_a$ denotes the
component-wise application of the inverse to $\frac{\partial e}{\partial
x}(a_i,-)$ for each $i$. 
\end{theorem*}

This theorem has a number of consequences. For now, let us name just three.
The first is that we may train a neural network by using the training data on
the whole network to create training data for each subunit, and then training
each subunit separately. To some extent this is well-known: it is responsible
for speedups due to backpropagation, as one never needs to compute the
derivatives of the function defined by the entire network. However the fact
that this functor is symmetric monoidal shows that we can vary the
backpropagation algorithm to factor the neural network into richer sub-parts
than simply carving it layer by layer.

Second, it gives a sufficient condition---which is both straightforward and
general---under which an error function works well under backpropagation.

Finally, it shows that backpropagation can be applied far more generally than
just to neural networks: it is compositional for all differentiable
parametrised functions. As a consequence, it shows that backpropagation gives a
sound method for computing gradient descent even if we introduce far more
general elements into neural networks than the traditional composites of linear
functions and activation functions.

\subsection*{Overview}
In \cref{sec.learners}, we define the category $\lrn$ of learning algorithms.
We present the main theorem in \cref{sec.functors}: given a
choice of error function and step size, gradient descent and backpropagation
give a functor from the category of parametrised functions to the category of
learning algorithms. In \cref{sec.nets}, we broaden this view to show how it
relates to neural networks. Next, in \cref{sec.bimonoids}, we note that the
category $\lrn$ has additional structure beyond just that of a symmetric
monoidal category: it has bimonoid structures that allow us to split and merge
connections to form networks. We also show this is useful in understanding the
construction of individual neurons, and in weight tying and convolutional neural
nets. We then explicitly compute an example of functoriality from neural nets to
learning algorithms (\textsection\ref{sec.examples}), and discuss implications
for this framework (\textsection\ref{sec.discuss}). The extended version
\cite{FST} of this article provides appendices with more technical aspects of
the proof of the main theorem, and a brief, diagram-driven introduction to
relevant topics in category theory.

\section{The category of learners}\label{sec.learners}
In this section we define a symmetric monoidal category $\lrn$ that
models supervised learning algorithms and their composites.
See extended version \cite{FST} for background on categories and string diagrams.

\begin{definition}
Let $A$ and $B$ be sets. 
A \emph{supervised learning algorithm}, or simply \emph{learner}, $A \to B$ is a tuple
$(P,I,U,r)$ where $P$ is a set, and $I$, $U$, and $r$ are functions of types:
\begin{align*}
I&\colon P \times A \to B, \\
U&\colon P \times A \times B \to P,\\
r&\colon P \times A \times B \to A.
\end{align*}
\end{definition}

We call $P$ the \emph{parameter space}; it is just a set. The map $I$
\emph{implements} a parameter value $p\in P$ as a function $I(p,-)\colon A
\to B$. We think of a pair $(a,b) \in A \times B$ as a \emph{training datum};
it pairs an input $a$ with an output $b$. The map $U\colon P \times A \times
B \to P$ is the \emph{update function}; given a `current' parameter $p$ and a
training datum $(a,b) \in A \times B$, it produces an `updated' parameter
$U(p,a,b) \in P$. This can be thought of as the learning step. The idea is
that the updated function $I(U(p,a,b),-)\colon A \to B$ would hopefully send
$a$ closer to $b$ than the function $I(p,-)$ did, though this is not a
requirement and is certainly not always true in practice. Finally, we have
the \emph{request} function $r\colon P \times A \times B \to A$. This takes
the same datum and produces a `requested value' $r(p,a,b) \in A$. The idea is
that this value will be sent to upstream learners for their own training.

\begin{remark}
The request function is perhaps a little mysterious at this stage. Indeed, it is
superfluous to the definition of a stand-alone learning algorithm: all we need
for learning is a space $P$ of functions $I(p,-)$ to search over, and a rule $U$ for
updating our parameter $p$ in light of new information. As we emphasised in the
introduction, the request function is crucial in \emph{composing} learning
algorithms: there is no composite update rule without the request function.

Another way to understand the role of the request function comes from
experiments in machine learning. Fixing some parameter $p$ and hence a function
$I(p,-)$, the request function allows us to choose a desired output $b$, and
then for any input $a$ return a new input $a':=r(p,a,b)$. In the case of
backpropagation, we will see we then have the intuition that $I(p,a')$ is closer
to $b$ than $I(p,a)$ is. For example, if we are classifying images, and $b$ is the
value indicating the classification `cat', then $a'$ will be a more `cat-like'
version of the image $a$. This is similar in spirit to what has been termed
inversion or `dreaming' in neural nets \cite{MV}.

A third way to understand the request function is by analogy with other
compositional structures: the request function plays an analogous role to the
put function in an asymmetric lens \cite{BPV} and the coplay function in an open
game \cite{GHWZ}.
\end{remark}

\begin{remark}
Using string diagrams\footnote{String diagrams are an alternative, but
nonetheless still formal, syntax for morphisms in a monoidal category. See extended version \cite{FST} for more details.} in $(\set,\times)$, we can draw an
implementation function $I$ as follows: 
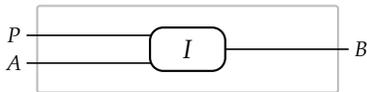
\[
\begin{tikzpicture}[oriented WD]
	\node[bb port sep=1, bb={2}{1}]                            (I)     {$I$};
	\node[ibb={2}{1}, fit=(I)]                          (outer) {};
	\node at ($(outer_in1'|-I_in1)-(0.3,0)$) {\footnotesize $P$};
	\node at ($(outer_in2'|-I_in2)-(0.3,0)$) {\footnotesize $A$};
	\node at ($(outer_out1')+(0.3,0)$) {\footnotesize $B$};
	\draw (outer_in1|-I_in1) to (I_in1);
	\draw (outer_in2|-I_in2) to (I_in2);
	\draw (I_out1') to (outer_out1);
\end{tikzpicture}
\]
One can do the same for $U$ and $r$, though we find it convenient to combine them into a single update--request function $(U,r)\colon P \times A \times B \to P
\times A$. This function can be drawn as follows:
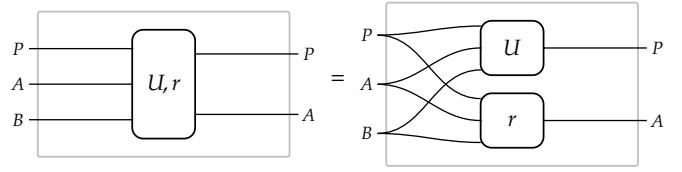
\[
\resizebox{.5\textwidth}{!}{
\begin{tikzpicture}[oriented WD, baseline=(current bounding box.center)]
	\node[bb port sep=2, bb={3}{2}]                            (U)     {$U,r$};
	\node[ibb={3}{2}, fit=(U)]                          (outer) {};
	\node at ($(outer_in1'|-U_in1)-(0.3,0)$) {\footnotesize $P$};
	\node at ($(outer_in2'|-U_in2)-(0.3,0)$) {\footnotesize $A$};
	\node at ($(outer_in3'|-U_in3)-(0.3,0)$) {\footnotesize $B$};
	\node at ($(outer_out1'|-U_out1)+(0.3,0)$) {\footnotesize $P$};
	\node at ($(outer_out2'|-U_out2)+(0.3,0)$) {\footnotesize $A$};
	\draw (outer_in1|-U_in1) to (U_in1);
	\draw (outer_in2|-U_in2) to (U_in2);
	\draw (outer_in3|-U_in3) to (U_in3);
	\draw (U_out1') to (outer_out1|-U_out1);
	\draw (U_out2') to (outer_out2|-U_out2);
\end{tikzpicture}
=
\begin{tikzpicture}[oriented WD,baseline=(current  bounding  box.center)]
	\node[bb port sep=1, bb={3}{1}]                            (U)     {$U$};
	\node[bb port sep=1, bb={3}{1},below=1 of U]                            (r)     {$r$};
	\node[ibb={3}{2}, fit=(U) (r)]                          (outer) {};
	\node at ($(outer_in1')-(0.3,0)$) {\footnotesize $P$};
	\node at ($(outer_in2')-(0.3,0)$) {\footnotesize $A$};
	\node at ($(outer_in3')-(0.3,0)$) {\footnotesize $B$};
	\node at ($(outer_out1'|-U_out1)+(0.3,0)$) {\footnotesize $P$};
	\node at ($(outer_out2'|-r_out1)+(0.3,0)$) {\footnotesize $A$};
	\draw (outer_in1) to (U_in1);
	\draw (outer_in2) to (U_in2);
	\draw (outer_in3) to (U_in3);
	\draw (outer_in1) to (r_in1);
	\draw (outer_in2) to (r_in2);
	\draw (outer_in3) to (r_in3);
	\draw (U_out1') to (outer_out1|-U_out1);
	\draw (r_out1) to (outer_out2|-r_out1);
\end{tikzpicture}
}
\]
\end{remark}

Let $(P,I,U,r)$ and $(P',I',U',r')$ be learners of the type $A\to B$. We
consider them to be equivalent if there is a bijection $f\colon P\to P'$ such that the following hold for each $p\in P$, $a\in A$, and $b\in B$:
\begin{align*}
  I'(f(p),a)&= I(p,a),\\
  U'(f(p),a,b) &= f(U(p,a,b)), \\ 
  r'(f(p),a,b) &= r(p,a,b).
\end{align*}
In fact, a stronger notion of equivalence---the equivalence relation generated by the existence of a surjection $f$ with these properties---also makes semantic sense, but we use this definition as it gives rise to faithfulness in the main theorem (\cref{thm.main}).
\begin{proposition}\label{prop.learners}
There exists a symmetric monoidal category $\lrn$ whose objects are sets and
whose morphisms are equivalence classes of learners.
\end{proposition}

The proof of \cref{prop.learners} is given in Appendix A of the extended version \cite{FST}. For now, we
simply specify the composition, identities, monoidal product, and braiding for
this symmetric monoidal category. Note that although we write in terms of
representatives, each of these is well defined, respecting the equivalence
relation on learners.

\paragraph{Composition}
Suppose we have a pair of learners
\[
  A \xrightarrow{(P,I,U,r)} B \xrightarrow{(Q,J,V,s)} C.
\]
The composite learner $A \to C$ is defined to be $(P\times Q,\, I \ast J,\, U
\ast V,\, r\ast s)$, where the implementation function is
\[
  (I \ast J)(p,q,a) \coloneqq J(q,I(p,a))
\]
the update function is
\[
  (U\ast V)(p,q,a,c) \coloneqq \bigg(U\big(p,a,s(q,I(p,a),c)\big),
  V\big(q,I(p,a),c\big)\bigg),
\]
and the request function
\[
  (r\ast s)(p,q,a,c) \coloneqq r\big(p,a,s(q,I(p,a),c)\big).
\]

Let us also present the composition rule using string diagrams in
$(\set,\times)$. Given learners $(P,I,U,r)$ and $(Q,J,V,s)$ as above, the composite
implementation function can be written as
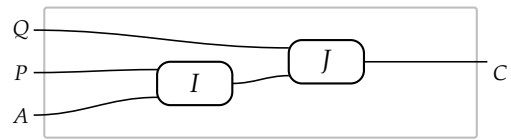
\[
\begin{tikzpicture}[oriented WD]
	\node[bb port sep=1, bb={2}{1}]                            (I)     {$I$};
	\node[bb port sep=1, bb={2}{1}, above right=-1 and .5 of I](J)     {$J$};
	\node[ibb={3}{1}, fit=(I) (J)]                          (outer) {};
	\node at ($(outer_in1')-(0.3,0)$) {\footnotesize $Q$};
	\node at ($(outer_in2')-(0.3,0)$) {\footnotesize $P$};
	\node at ($(outer_in3')-(0.3,0)$) {\footnotesize $A$};
	\node at ($(outer_out1')+(0.3,0)$) {\footnotesize $C$};
	\draw (outer_in1) to (J_in1);
	\draw (outer_in2) to (I_in1);
	\draw (outer_in3) to (I_in2);
	\draw (I_out1') to (J_in2);
	\draw (J_out1') to (outer_out1|-J_out1);
\end{tikzpicture}
\]
while the composite update--request function $(U \ast V, r \ast s)$ can be
written as:
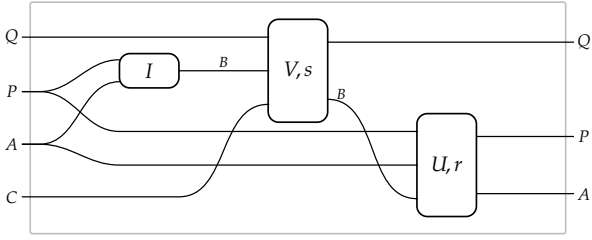
\[ 
  \resizebox{.45\textwidth}{!}{
\begin{tikzpicture}[oriented WD]
	\node[bb port sep=1, bb={2}{1}]                           (I)    {$I$};
	\node[bb port sep=2, bb={3}{2}, above right=-4 and 1 of I](V)   {$V,s$};
	\node[bb port sep=2, bb={3}{2}, below right=-0.5 and 1 of V](U)	{$U,r$};
	\node[ibb={4}{3}, fit=(I) (V) (U)]                          (outer) {};
	\begin{scope}[font=\footnotesize]
  	\node at ($(outer_in1'|-V_in1)-(0.3,0)$) {$Q$};
  	\node at ($(outer_in2')-(0.3,0)$) {$P$};
  	\node at ($(outer_in3')-(0.3,0)$) {$A$};
  	\node at ($(outer_in4')-(0.3,0)$) {$C$};
  	\node at ($(outer_out1'|-V_out1)+(0.3,0)$) {$Q$};
  	\node at ($(outer_out2'|-U_out1)+(0.3,0)$) {$P$};
  	\node at ($(outer_out3'|-U_out2)+(0.3,0)$) {$A$};
  	\draw (outer_in1|-V_in1) to (V_in1);
  	\draw (outer_in2) -- ($(outer_in2)+(.2,0)$) to (I_in1);
  	\draw let \p1=(I.south west), \p2=($(U_in1)$), \n1=\bbportlen in
  		(outer_in2) -- ($(outer_in2)+(.2,0)$) to (\x1-\n1, \y2) -- (U_in1);
  	\draw (outer_in3) -- ($(outer_in3)+(.2,0)$) to (I_in2);
  	\draw let \p1=(I.south west), \p2=($(U_in2)$), \n1=\bbportlen in
  		(outer_in3) -- ($(outer_in3)+(.2,0)$)to (\x1-\n1, \y2) -- (U_in2);
  	\draw let \p1=(I.south east), \p2=($(outer_in4)$), \n1=\bbportlen in
  		(outer_in4) to (\x1+\n1, \y2) to (V_in3);
  	\draw (V_out1) to (outer_out1|-V_out1);
  	\draw (V_out2) to node[above=-2pt, pos=.1] {\scriptsize $B$} (U_in3);
  	\draw (U_out1) to (outer_out2|-U_out1);
  	\draw (U_out2) to (outer_out3|-U_out2);
  	\draw (I_out1) to node[above=-2pt] {\scriptsize $B$} (V_in2);
	\end{scope}
\end{tikzpicture}
}
\]
Here the splitting represents the diagonal map $A \to A \times
A$, i.e.\ $a \mapsto (a,a)$.

We hope that the reader might find visually tracing
through these diagrams helpful for making sense of the composition rule. To repeat the intuition from the introduction, suppose given current parameters $p\in P$ and $q\in Q$ and training data $a\in A$ and $c\in C$. $I$ takes $p$ and $a$ and produces some $b\in B$ for training the second component. Along with $q$ and $c$, $b$ is used to compute an updated parameter $q'$ together with a value $b'$ for training the first component. Along with $p$ and $a$, $b'$ is used to compute an updated parameter $p'$ together with a value $a'$.

\paragraph{Identities}
For each object $A$, we have the identity map
\[
  (\rr^0, \id,!,\pi_2)\colon A \longrightarrow A,
\]
where $\id\colon \rr^0 \times A  \to A$ is the
second projection (as this is a bijection, we abuse notation to write this
projection as $\id$),\: $!\colon \rr^0 \times A \times A \to \rr^0$ is the unique
function, and $\pi_2\colon \rr^0 \times A \times A \to A$ is the projection
onto the final factor (again, ignoring the $\rr^0$).

\paragraph{Monoidal product} The monoidal product of objects $A$ and $B$ is
simply their cartesian product $A \times B$ as sets. The monoidal product of
morphisms $(P,I,U,r)\colon A \to B$ and $(Q,J,V,s)\colon C \to D$ is defined to
be $(P\times Q,\,I\prl J,\,U\prl V,\,r\prl s)$, where the implementation function
is
\[
  	(I\prl J)(p,q,a,c) \coloneqq (I(p,a),J(q,c))
\]
the update function is
\[
	(U\prl V)(p,q,a,c,b,d) \coloneqq (U(p,a,b),V(q,c,d)) 
\]
and the request function is
\[
(r\prl s)(p,q,a,c,b,d) \coloneqq (r(p,a,b),s(q,c,d)).
\]
We use the notation $\;\prl\;$ because monoidal product can be thought of as
parallel---rather than series---composition.

We also present this in string diagrams:
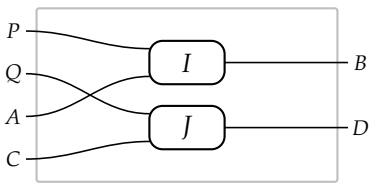
\[
\begin{tikzpicture}[oriented WD]
	\node[bb port sep=1, bb={2}{1}]                            (I)     {$I$};
	\node[bb port sep=1, bb={2}{1}, below= of I]                  (J)     {$J$};
	\node[ibb={4}{2}, fit=(I) (J)]                          (outer) {};
	\node at ($(outer_in1')-(0.3,0)$) {\footnotesize $P$};
	\node at ($(outer_in2')-(0.3,0)$) {\footnotesize $Q$};
	\node at ($(outer_in3')-(0.3,0)$) {\footnotesize $A$};
	\node at ($(outer_in4')-(0.3,0)$) {\footnotesize $C$};
	\node at ($(outer_out1'|-I_out1)+(0.3,0)$) {\footnotesize $B$};
	\node at ($(outer_out2'|-J_out1)+(0.3,0)$) {\footnotesize $D$};
	\draw (outer_in1) to (I_in1);
	\draw (outer_in2) to (J_in1);
	\draw (outer_in3) to (I_in2);
	\draw (outer_in4) to (J_in2);
	\draw (I_out1) to (outer_out1|-I_out1);
	\draw (J_out1) to (outer_out2|-J_out1);
\end{tikzpicture}
\]

\[
  \resizebox{.25\textwidth}{!}
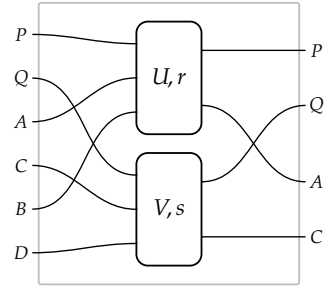
{
\begin{tikzpicture}[oriented WD,scale=.7]
	\node[bb port sep=2, bb={3}{2}]                            (U)     {$U,r$};
	\node[bb port sep=2, bb={3}{2}, below= of U]               (V)     {$V,s$};
	\node[ibb={6}{4}, fit=(U) (V)]                          (outer) {};
	\node at ($(outer_in1')-(0.35,0)$) {\footnotesize $P$};
	\node at ($(outer_in2')-(0.35,0)$) {\footnotesize $Q$};
	\node at ($(outer_in3')-(0.35,0)$) {\footnotesize $A$};
	\node at ($(outer_in4')-(0.35,0)$) {\footnotesize $C$};
	\node at ($(outer_in5')-(0.35,0)$) {\footnotesize $B$};
	\node at ($(outer_in6')-(0.35,0)$) {\footnotesize $D$};
	\node at ($(outer_out1'|-U_out1)+(0.35,0)$) {\footnotesize $P$};
	\node at ($(outer_out2'|-U_out2)+(0.35,0)$) {\footnotesize $Q$};
	\node at ($(outer_out3'|-V_out1)+(0.35,0)$) {\footnotesize $A$};
	\node at ($(outer_out4'|-V_out2)+(0.35,0)$) {\footnotesize $C$};
	\draw (outer_in1) to (U_in1);
	\draw (outer_in2) to (V_in1);
	\draw (outer_in3) to (U_in2);
	\draw (outer_in4) to (V_in2);
	\draw (outer_in5) to (U_in3);
	\draw (outer_in6) to (V_in3);
	\draw (U_out1) to (outer_out1|-U_out1);
	\draw (U_out2) to (outer_out3|-V_out1);
	\draw (V_out1) to (outer_out2|-U_out2);
	\draw (V_out2) to (outer_out4|-V_out2);
\end{tikzpicture}
	}
\]

\paragraph{Braiding}
A symmetric braiding $A \times B \to B \times A$ is given by $(\rr^0, \sigma,
!, \sigma\circ \pi)$ where $\sigma\colon A\times B\to B\times A$ is the
usual swap function $(a,b)\mapsto(b,a)$ and $\pi\colon \rr^0 \times (A
\times B) \times (B \times A) \to B \times A$ is again the projection onto
the final factor.

A proof that this is a well-defined symmetric monoidal category can be
found in the extended version \cite{FST}.


\section{Gradient descent and backpropagation}\label{sec.functors}
In this section we show that gradient descent and backpropagation define a
strong symmetric monoidal functor from a symmetric monoidal category $\para$, of
differentiable parametrised functions between finite dimensional Euclidean
spaces, to the symmetric monoidal category $\lrn$ of learning algorithms.

We first define the category of differentiable parametrised functions. A
\emph{Euclidean space} is one of the form $\rr^n$ for some $n \in \nn$. We call $n$ the
\emph{dimension} of the space, and write an element $a \in \rr^n$ as
$(a_1,\dots,a_n)$, or simply $(a_i)_i$, where each $a_i \in \rr$. 
.
For Euclidean spaces $A=\rr^n$ and $B=\rr^m$, define a \emph{differentiable
parametrised function} $A \to B$ to be a pair $(P,I)$, where $P$ is a Euclidean space and
$I\colon P\times A\to B$ is a differentiable function. We call two such pairs $(P,I)$, $(P',I')$ equivalent if there exists a differentiable bijection $f\colon P \to P'$ such that for all $p \in P$ and $a \in A$ we have $I'(f(p),a) = I(p,a)$. Differentiable
parametrised functions between Euclidean spaces form a symmetric monoidal
category.

\begin{definition}
We write $\para$ for the strict symmetric monoidal category whose objects are
Euclidean spaces and whose morphisms $\rr^n \to \rr^m$ are equivalence classes
of differentiable parametrised functions $\rr^n \to \rr^m$. 

Composition of $(P,I)\colon \rr^n \to \rr^m$ and $(Q,J)\colon \rr^m \to \rr^\ell$ is
given by $(P \times Q,I \ast J)$ where
\[
(I \ast J)(p,q,a) = J(q,I(p,a)).
\]
The monoidal product of objects $\rr^n$ and $\rr^m$ is the object $\rr^{n+m}$,
while the monoidal product of morphisms $(P,I)\colon \rr^n \to \rr^m$ and $(Q,J)\colon
\rr^\ell \to \rr^k$ is given by $(P \times Q,I \prl J)$ where
\[
(I \prl J)(p,q,a,c) = \big(I(p,a),J(q,c)\big).
\]
The braiding $\rr^n \prl \rr^m \to \rr^m \prl \rr^n$ is given by $(\rr^0,\sigma)$ where $\sigma(a,b) =
(b,a)$.
\end{definition}
It is straightforward to check this is a well defined symmetric monoidal
category. We are now in a position to state the main theorem.

\begin{theorem} \label{thm.main}
Fix a real number $\ve > 0$ and $e(x,y)\colon \rr \times \rr \to \rr$
differentiable such that $\frac{\partial e}{\partial x}(x_0,-)\colon \rr \to \rr$
is invertible for each $x_0 \in \rr$.  Then we can define a faithful,
injective-on-objects, strong symmetric monoidal functor
\[
  L_{\ve,e}\colon\para \longrightarrow \lrn
\]
that sends each parametrised function $I\colon P \times A \to B$ to the learner
$(P,I,U_I,r_I)$ defined by
\[
  U_I(p,a,b) \coloneqq p - \varepsilon\nabla_pE_I(p,a,b)
\]
and
\[
  r_I(p,a,b) \coloneqq f_a\bigg(\nabla_a E_I(p,a,b)\bigg),
\]
where $E_I(p,a,b) \coloneqq \sum_j e(I_j(p,a),b_j)$, and $f_a$ is component-wise
application of the inverse to $\frac{\partial e}{\partial x}(a_i,-)$ for each
$i$.
\end{theorem}
\begin{proof}[Proof (sketch)]
The proof of this theorem amounts to observing that the chain rule is functorial
given the above setting. The key points are the use of the chain rule to show
the functoriality of the $P$-part of the update function and the request
function. A full proof is given in the extended version \cite{FST}. 
\end{proof}

We call $\ve$ the \emph{step size}, $e$ the \emph{error function}, and $E_I$ the
\emph{total error (with respect to $I$)}. We also call the functors $L_{\ve,e}$, so named
because they turn parametrised functions into a \emph{learning} algorithms,
the \emph{gradient descent/backpropagation functors}.

\begin{remark}
The update function $U_I$ encodes what is known as \emph{gradient descent}: the
parameter $p$ is updated by moving it an $\ve$-step in the direction that most
reduces the total error $E_I$. 

The request function $r_I$ encodes the \emph{backpropagation} value, passing
back the gradient of the total error with respect to the input $a$, as modified
by the invertible function $f_a$. To pick an example, the functoriality of
$L_{\ve,e}$ says that the following two update functions are equal:
\begin{itemize}
\item The update function $U_{((I\,\prl\, J)\ast K) \ast M}$, which represents
gradient descent on the
composite of parametrised functions $((I\prl J)\ast K) \ast M$.
\item The update function $((U_I\prl U_J)\ast U_K) \ast U_M$, which represents
the composite, according to the structure in $\lrn$, of the update functions
$U_I$, $U_J$, $U_K$, and $U_M$ together with the request functions $r_I$,
$r_J$, $r_K$, and $r_M$.
\end{itemize}
This shows that we may compute gradient descent by local computation of the
gradient together with local message passing. This is the backpropagation
algorithm.
\end{remark}

\begin{example}[Quadratic error] \label{ex.quaderror}
Quadratic error is given by the error function $e(x,y)
\coloneqq\tfrac12(x-y)^2$,  so that the total error is given by
\[
E_I(p,a,b) = \tfrac12 \sum_j (I_j(p,a) -b_j)^2 = \tfrac12 \norm{I(p,a)-b}^2.
\] 
In this case $\frac{\partial e}{\partial x}(x_0,-)$ is the function $y\mapsto
x_0-y$. This function is its own inverse, so we have $f_{x_0}(y)\coloneqq x_0-y$.

Fixing some step size $\ve >0$, we have 
\begin{align*}
U_I(p,a,b) &\coloneqq p - \varepsilon\nabla_pE_I(p,a,b) \\
&= \Big(p_k-\ve\sum_j(I_j(p,a)-b_j)\tfrac{\partial I_j}{\partial p_k}(p,a)\Big)_k
\end{align*}
and similarly
\begin{align*}
r_I(p,a,b) &\coloneqq a - \nabla_aE_I(p,a,b)\\
&=\Big(a_i - \sum_j(I_j(p,a)-b_j)\tfrac{\partial I_j}{\partial
a_i}(p,a)\Big)_i.
\end{align*}
Thus given this choice of error function, the functor $L_{\ve,e}$ of \cref{thm.main}
just implements, as update function, the usual gradient descent with step size
$\ve$ with respect to the quadratic error.
\end{example}

\begin{remark}
Comparing the requests $r_I$ to the updates $U_I$ in \cref{ex.quaderror}, one may
notice that they are similar, except that the former seem to be missing the $\ve$.
One might wonder why the two are different or where the $\ve$ factor has gone.

\cref{thm.main} shows, however, that in fact the similarity between $r_I$ and $U_I$
is something of a coincidence. What is
important about requests, and hence the messages passed backward in
backpropagation, is the fact that they are constructed by inverting certain partial
derivatives which are then applied to the gradient of the total error with
respect to the input. We interpret the result as a `corrected' input value that, if used instead,
would reduce the total error with respect to the given output and current
parameter value. In particular, the resemblance of the request values to gradient
descent in \cref{ex.quaderror} is just an artifact of the choice of quadratic
error function $e(x,y) \coloneqq\tfrac12(x-y)^2$.
\end{remark}

\section{From networks to parametrised functions} \label{sec.nets}
In the previous section we showed that gradient descent and backpropagation---for a given choice of error function and step size---define a
functor from differentiable parametrised functions to supervised learning
algorithms. But backpropagation is often considered an algorithm executed on a
neural net. How do neural nets come into the picture? As we shall see, neural
nets are a method for defining parametrised functions from network architectures.

This method, like backpropagation itself, is also compositional---it respects
the gluing together of neural networks. To formalise this, we first define a
category $\nnet$ of neural networks. Implementation of each neural net will then
define a parametrised function, and in fact we get a functor
$I\colon\nnet\to\para$. Note that just as defining a gradient
descent/backpropagation functor depends on a choice (namely, of error function
and step size), so too does defining $I$.  Namely, we must choose an activation
function.

Recall from the introduction that a neural network layer of type $(m,n)$ is a
subset of $[m] \times [n]$, where $m, n \in \nn$ and $[n]=\{1,\dots,n\}$. A
\emph{$k$-layer neural network of type $(m,n)$} is a sequence of neural network
layers of types $(n_0,n_1)$, $(n_1,n_2)$, $\dots$, $(n_{k-1},n_k)$, where
$n_0=m$ and $n_k=n$. A \emph{neural network of type $(m,n)$} is a $k$-layer
neural network for some $k$.

Given a neural network of type $(m,n)$ and a neural network of type $(n,p)$ we
may concatenate them to get a neural network of type $(m,p)$. Note that when $m=n$, we consider the $0$-layer neural network to be a morphism. Concatenating any neural network on either side with the $0$-layer neural network does not change it.

\begin{definition}
The category $\nnet$ of neural networks has as objects natural numbers and as
morphisms $m \to n$ neural networks of type $(m,n)$. Composition is given by
concatenation of neural networks. The identity morphism on $n$ is the $0$-layer neural network.
\end{definition}

Since composition is just concatenation it is immediately
associative, and we have indeed defined a category.

\begin{proposition}\label{prop.nnet_para}
Given a differentiable function $\sigma\colon \rr \to \rr$, we have a functor
\[
I^\sigma\colon \nnet \longrightarrow \para.
\]
On objects, $I^\sigma$ maps each natural number $n$ to the $n$-dimensional Euclidean
space $\rr^n$. 

On morphisms, each 1-layer neural network $C\colon m \to n$ is mapped to the
parametrised function
\begin{align*}
I^\sigma_C\colon \rr^{\lvert C\rvert+n} \times \rr^m &\longrightarrow \rr^n; \\
\big((w_{ji},w_j),x_i\big)_{\substack{1 \le i \le m \\ 1 \le j \le n}} &\longmapsto \Bigg(\sigma\bigg(\sum_{i}
w_{ji}x_i+w_j\bigg)\Bigg)_{1 \le j \le n}.
\end{align*}
Given a neural net $N = C_1, \dots, C_k$, the image under $I^\sigma$ is the composite
of the image of each layer:
\[
I^\sigma_N = I^\sigma_{C_1} \ast \dots \ast I^\sigma_{C_k}
\]
\end{proposition}

We call $\sigma$ the \emph{activation function}, the $w_{ji}$
\emph{weights}, where $(i,j) \in C$, and the $w_j$ \emph{biases}, where $j
\in [n]$.

\begin{proof}
The proof of this proposition is straightforward. Note in particular that the
image $I^\sigma_C$ of each layer $C$ is differentiable, and so their composites are
too. Also note that the image of a $0$-layer neural net is the empty composite,
so identities map to identities.  Composition is preserved by definition. 
\end{proof}

Composing an implementation functor $I^\sigma$ with a gradient descent/backpropagation
functor $L_{\ve,e}$, we get a functor 
\[
  \begin{tikzcd}[row sep=1.5ex]
\nnet \ar[rr] \ar[dr, "I^\sigma" below] & & \lrn \\
& \para \ar[ur, "L_{\ve,e}" below]
\end{tikzcd}
\]
This states that, given choices of activation function $\sigma$, error function $e$, and step
size $\ve$, a neural net defines a supervised learning algorithm, and does so in a
compositional way.

A symmetric monoidal structure, both on $\nnet$ and on the above functors, can
be given by generalising the category $\nnet$ to the category where morphisms
are directed acyclic graphs with interfaces; details on such a category can be
found in \cite{FC}, see also \cref{rem.idag}.

In Section \ref{sec.examples}, we will compute an extended example of the use of neural nets to compositionally define supervised learning algorithms. Before
this, however, we discuss additional compositional structure available to us in
$\lrn$, $\para$, and the aforementioned
monoidal generalisation of $\nnet$.

\section{Networking in $\lrn$}\label{sec.bimonoids}
Our formulation of supervised learning algorithms as morphisms in a monoidal
category means learning algorithms can be formed by combining other learning
algorithms in sequence and in parallel. In fact, as hinted at by neural networks
themselves, more structure is available to us: we are able to form new learning
algorithms by combining others in networks of learners where wires can split and
merge. Formally, this means each object in the category of learners is equipped
with the structure of a bimonoid.

For this, note first that the symmetric monoidal category $\fvect$ of linear
maps between Euclidean spaces sits inside the category $\para$ of parametrised
functions; we simply consider each linear map as parametrised by the trivial
parameter space $\rr^0$.  Given a choice of step size and error function, and
hence a functor $L_{\ve,e}\colon \para \to \lrn$ as in \cref{thm.main}, we thus have an inclusion 
\[
  \fvect\hookrightarrow \para \stackrel{L_{\ve,e}}\hookrightarrow \lrn.
\]
This allows us to construct a learning algorithm---that is, a morphism in
$\lrn$---as the image of any morphism in $\fvect$, and the output algorithms
obey the same equations as the input linear maps. In particular, from graphical
linear algebra \cite{BE,BSZ} we know that each object in $\fvect$ is equipped
with a bimonoid structure, so we can use our functor $L_{\ve,e}$ to equip each
object of the form $\rr^n$ in $\lrn$ with a bimonoid structure. This bimonoid
structure is what makes the neural network notation feasible: we can interpret
the splitting and combining in a way coherent with composition.

In fact, the bimonoids constructed depend only on the choice of error function;
we need not specify the step size. As an example, we detail the construction
using backpropagation with respect to the quadratic error (\cref{ex.quaderror}).

\begin{proposition} \label{prop.bimonoids}
Gradient descent with respect to the quadratic error and step size $\ve$
defines a symmetric monoidal functor $\fvect \to \lrn$. This implies each object
in the image of this functor can be equipped with the structure of a bimonoid.
\end{proposition}

Explicitly, the bimonoid maps are given as follows. Note they all have
trivial parameter space $\rr^0$, so we denote the unique update function $!\colon\rr^0 \times A \times B \to\rr^0$. 

\smallskip
\noindent
\begin{tabular}{ |>{\centering\arraybackslash}m{.14\textwidth}
|>{\centering\arraybackslash}m{.13\textwidth} |>{\centering\arraybackslash}m{.13\textwidth} |}
\hline
& Implementation & Request \\
\hline
Multiplication~$\mu$ \newline
$
(1,I_\mu,!,r_\mu)
$
$
\begin{tikzpicture}[oriented WD]
	\node[ibb={2}{1}]                          (outer) {};
	\node at ($(outer_in1')-(0.3,0)$) {\footnotesize $A$};
	\node at ($(outer_in2')-(0.3,0)$) {\footnotesize $A$};
	\node at ($(outer_out1')+(0.3,0)$) {\footnotesize $A$};
	\draw (outer_in1) -- (outer_in1') to (outer_out1') -- (outer_out1);
	\draw (outer_in2) -- (outer_in2') to (outer_out1') -- (outer_out1);
\end{tikzpicture}
$
&
$I_\mu(a_1,a_2)=a_1 +a_2$
& 
$r_\mu((a_1,a_2),a_3)=(a_3-a_2,a_3-a_1)$
\\
\hline
Unit $\eta$

$
(1,I_\eta,!,r_\eta)
$
$
\begin{tikzpicture}[oriented WD]
	\node[ibb={0}{1}]                          (outer) {};
	\node at ($(outer_out1')+(0.3,0)$) {\footnotesize $A$};
	\node[draw, fill, circle, inner sep=1pt] (u) at ($(outer_out1)-(.45,0)$) {};
	\draw (u) to (outer_out1);
\end{tikzpicture}
$
&
$I_\eta(0)=0$
&
$r_\eta(a)=0$
\\
\hline
\small
Comultiplication~$\delta$
$
(1,I_\delta,!,r_\delta)
$
$
\begin{tikzpicture}[oriented WD]
	\node[ibb={1}{2}]                          (outer) {};
	\node at ($(outer_in1')-(0.3,0)$) {\footnotesize $A$};
	\node at ($(outer_out1')+(0.3,-0.2)$) {\footnotesize $A$};
	\node at ($(outer_out2')+(0.3,0.2)$) {\footnotesize $A$};
	\draw (outer_in1) -- (outer_in1') to ($(outer_out1')+(0,-0.2)$) --
	($(outer_out1)+(0,-0.2)$);
	\draw (outer_in1) -- (outer_in1') to ($(outer_out2')+(0,0.2)$) --
	($(outer_out2)+(0,0.2)$);
\end{tikzpicture}
$
&
$I_\delta(a)= (a,a)$
&
$r_\delta(a_1,(a_2,a_3))=a_2+a_3 -a_1$
\\
\hline
Counit $\epsilon$
$
(1,I_\epsilon,!,r_\epsilon)
$
$
\begin{tikzpicture}[oriented WD]
	\node[ibb={1}{0}]                          (outer) {};
	\node at ($(outer_in1')-(0.3,0)$) {\footnotesize $A$};
	\node[draw, fill, circle, inner sep=1pt] (u) at ($(outer_in1)+(.45,0)$) {};
	\draw (outer_in1) to (u);
\end{tikzpicture}
$
&
$I_\epsilon(a)=0$
&
$r_\epsilon(a)= 0$
\\
\hline
\end{tabular}
\smallskip

\begin{remark} 
We actually have many different bimonoid structures in $\lrn$: each choice of
error function defines one, and these are often distinct. For example, if we
choose $e(x,y) = xy$ then the request function on the multiplication is
instead given by $r'_\mu(a_1,a_2,a_3) = (a_3,a_3)$ and the request function
on the comultiplication is instead given by $r'_\delta(a_1,a_2,a_3) =
a_2+a_3$. While this is a rather strange error function---minimising error
entails sending outputs to $0$---the existence of such structures is
interesting.
\end{remark}

A choice of bimonoid structures, such as that given by \cref{prop.bimonoids},
allows us to interpret network diagrams in the monoidal category $(\lrn,\,\prl\,)$ from \cref{prop.learners}. Indeed, they give canonical
interpretations of splitting, joining, initializing, and discarding wires.

\begin{example}[Building neurons]
As learning algorithms implemented with respect to quadratic error (see \cref{ex.quaderror}) and some step size $\ve$,
neural networks have a rather simple structure: they are generated by three
basic learning algorithms---scalar multiplication $\lambda$, bias $\beta$, and
an activation function $\sigma$---together with the bimonoid multiplication $\mu$ and
comultiplication $\delta$ given by \cref{prop.bimonoids}. 

The scalar multiplication learning algorithm $\lambda\colon \rr \to \rr$, which
we shall represent graphically by the string diagram in $\lrn$\footnote{Note
  that these are string diagrams in $(\lrn,\,\prl\,)$, while the string diagrams
  of \cref{sec.learners} were string diagrams in $(\set,\times)$. As always, string
  diagrams represent morphisms in a category, with the domain at the left of the
  diagram and codomain on the right. For more details see the extended version \cite{FST}.
}
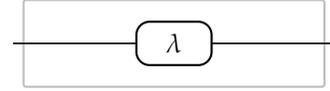
\[
\begin{tikzpicture}[oriented WD]
	\node[bb={1}{1}]		 (w)     {$\lambda$};
	\node[ibb={1}{1}, fit=(w)]                (outer) {};
	\draw (outer_in1) to (w_in1);
	\draw (w_out1) to (outer_out1);
\end{tikzpicture}
\]
is given by the parameter space $\rr$, implementation function $\lambda(w,x) =
wx$, update function $U_\lambda(w,x,y) = w-\ve x(wx-y)$, and request function
$r_\lambda(w,x,y) = x-w(wx-y)$.

The bias learning algorithm $\beta\colon \rr^0 \to \rr$, which we represent
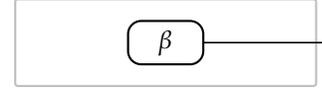
\[
\begin{tikzpicture}[oriented WD]
	\node[bb={0}{1}]		 (b)     {$\beta$};
	\node[ibb={1}{1}, fit=(b)]                (outer) {};
	\draw (b_out1) to (outer_out1);
\end{tikzpicture}
\]
is given by the parameter space $\rr$, implementation $\beta(w) = w$, update
function $U_\beta(w,y) = (1-\ve)w+\ve y$, and trivial request function, since it
has trivial input space.

The activation function learning algorithm $\sigma\colon \rr \to \rr$,
represented
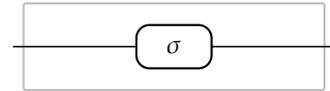
\[
\begin{tikzpicture}[oriented WD]
	\node[bb={1}{1}]		 (s)     {$\sigma$};
	\node[ibb={1}{1}, fit=(s)]                (outer) {};
	\draw (outer_in1) to (s_in1);
	\draw (s_out1) to (outer_out1);
\end{tikzpicture}
\]
has trivial parameter space, and is specified by some choice of activation
function $\sigma\colon \rr \to \rr$, together with the trivial update function
and the request function $r_\sigma(x,y) = x-(x-y)\tfrac{\partial
\sigma}{\partial x}(x)$. 

Then, every neuron in a neural network can be understood as a composite of these
generators as follows: first, a monoidal product of the required number of
scalar multiplication algorithms and a bias algorithm, then a composite of
$\mu$'s, an activation function, and finally a composite of $\delta$s.
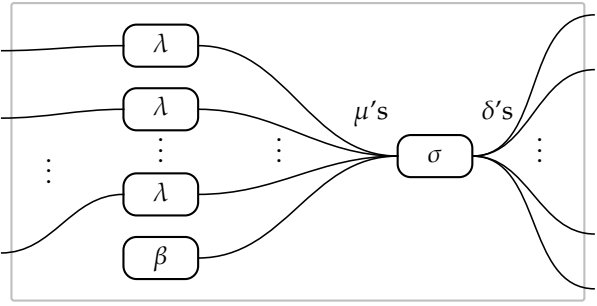
\[
  \resizebox{.45\textwidth}{!}{
\begin{tikzpicture}[oriented WD, bb port sep=1, text height=1.5ex, text depth=.5ex]
	\node[bb={1}{1}] (w1) {$\lambda$};
	\node[bb={1}{1}, below=1 of w1] (w2) {$\lambda$};
	\node[bb={1}{1}, below=2 of w2] (wn) {$\lambda$};
	\node at ($(w2.center)!.55!(wn.center)$) (vdots) {$\vdots$};
	\node[bb={0}{1}, below=1 of wn] (b) {$\beta$};
	\node[bb={1}{1}, right=2 of vdots]  (activ)	{$\sigma$};
	\node[ibb={4}{6}, fit=(w1) (b) (activ)]  (outer) {};
	\draw (outer_in1) to (w1_in1);
	\draw (outer_in2) to (w2_in1);
	\node at ($(outer_in3)!.3!(vdots)$) {$\vdots$};
	\node at ($(activ_in1)!.5!(vdots)$) {$\vdots$};
	\draw (outer_in4) to (wn_in1);
	\draw (w1_out1) to (activ_in1);
	\draw (w2_out1) to (activ_in1);
	\draw (wn_out1) to (activ_in1);
	\draw (b_out1) to (activ_in1);
	\node[above left=1 and 0 of activ_in1] {$\mu$'s};
	\node[above right=1 and 0 of activ_out1] {$\delta$'s};
	\foreach \i in {1,2,5,6} {
		\draw (activ_out1) to (outer_out\i);
	}
	\node at ($(activ_out1)!.55!(outer_out1|-activ_out1)$) {$\vdots$};
\end{tikzpicture}
}
\]
Composing these units using the composition rule in $\lrn$ further constructs
any learning algorithm that can be obtained by gradient descent and
backpropagation on a neural network with respect to the quadratic error.
\end{example}

\begin{example}[Weight tying]

Weight tying (or weight sharing) in neural networks is a method by which
parameters in different parts of the network are constrained to be equal. It is
used in convolutional neural networks, for example, to force the network to
learn the same sorts of basic shapes appearing in different parts of an image
\cite{LCB}.  This is easily represented in our framework. Before explaining how
this works, we first explain a way of factoring morphisms in $\para$ into basic
parts.

Morphisms in $\para$ are roughly generated by morphisms of two different types:
trivially-parametrised functions and parametrised constants. Given a
differentiable function $f\colon\rr^n\to\rr^m$, we consider it a \emph{trivially
parametrised function} $\rr^0\times\rr^n\to\rr^m$, whose parameter space
$P=\rr^0$ is a point.  By a \emph{parametrised constant}, we mean an identity
morphism $1_P\colon P\to P$, considered as a parametrised function $P \times
\rr^0\to P$.

In particular, every parametrised function can be written as a composite, using
the bimonoid structure, of a trivially parametrised function and a parametrised
constant. To see this, we use string diagrams in $(\para,\,\prl\,)$, where here we denote a parametrised function $I\colon P\times A\to B$ as a box
labeled $(P,I)$ with input $A$ and output $B$. It is easy to check that any
parametrised function $I\colon P\times A\to B$ is the composite of a trivially
parametrised function and a parametrised constant as follows
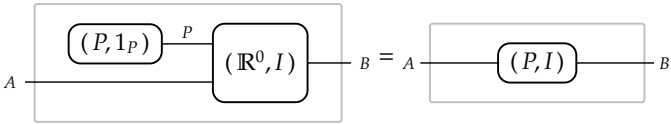
\[
\resizebox{.5\textwidth}{!}{
\begin{tikzpicture}[oriented WD, inner
  xsep=5pt, baseline=(current bounding box.center)]
	\node[bb={2}{1}] (0I) {$(\,\rr^0,I\,)$};
	\node[bb={0}{1}, left=.5 of 0I_in1] (Pid) {$(\,P,1_P\,)$};
	\node[ibb={0}{0}, bbx=.5cm, fit=(Pid) (0I)] (outer) {};
	\node (oi) at (outer.west|-0I_in2) {};
	\node (oo) at (0I_out1-|outer.east) {};
	\draw[shorten <=-4pt] (oi.center) -- (0I_in2);
	\draw (Pid_out1) -- (0I_in1);
	\draw[shorten >=-4pt] (0I_out1) -- (oo.center);
	\draw[label]
		node[left=2pt of oi] {$A$}
		node[above=2pt] at ($(Pid_out1)!.5!(0I_in1)$) {$P$}
		node[right=2pt of oo] {$B$}
	;
\end{tikzpicture}
=
\begin{tikzpicture}[oriented WD, inner xsep=5pt, baseline=(current bounding box.center)]
	\node[bb={1}{1}] (Pf) {$(\,P,I\,)$};
	\node[ibb={1}{1}, bbx=1cm, fit=(Pf)] (outer2) {};
	\draw (outer2_in1) -- (Pf_in1);
	\draw (Pf_out1) -- (outer2_out1);
	\draw[label]
		node[left=2pt of outer2_in1] {$A$}
		node[right=2pt of outer2_out1] {$B$}
	;
\end{tikzpicture}
}
\]

Since these morphisms are the same in $\para$, they correspond to the same
learning algorithm, by the functoriality of $L_{\ve,e}$, \cref{thm.main}. Looking at the right hand picture,
suppose given a training datum $(a,b)$. The $(\rr^0,I)$ block has trivial
parameter space, so updates on it do nothing; however, it is capable of sending
a request to the input $A$ and the $(P,1_P)$ block. The $(P,1_P)$ block then
performs the desired update. Again, the result of doing so must be the same, by
the main theorem.

This suggests how one should think of weight tying. The schematic idea,
represented in string diagrams, is as follows:
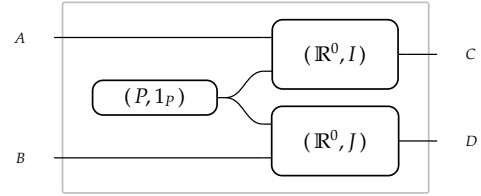
\[
  \resizebox{.35\textwidth}{!}{
\begin{tikzpicture}[oriented WD, bbx=1cm, bb port length=4pt, inner xsep=15pt]
	\node[bb={2}{1}] (I) {$(\,\rr^0,I\,)$};
	\node[bb={2}{1}, below=of I] (J) {$(\,\rr^0,J\,)$};
	\node[bb={0}{1}] at ($(I)!.5!(J)-(3,0)$) (P) {$(\,P,1_P\,)$};
	\node[ibb={0}{0}, bbx=.5cm, fit=(I) (J) (P)] (outer) {};
	\node (oi1) at (outer.west|-I_in1) {};
	\node (oi2) at (outer.west|-J_in2) {};
	\node (oo1) at (outer.east|-I_out1) {};
	\node (oo2) at (outer.east|-J_out1) {};
	\draw[shorten <=-4pt] (oi1.center) -- (I_in1);
	\draw[shorten <=-4pt] (oi2.center) -- (J_in2);
	\draw[shorten >=-4pt] (I_out1) -- (oo1.center);
	\draw[shorten >=-4pt] (J_out1) -- (oo2.center);
	\draw (P_out1) to (I_in2);
	\draw (P_out1) to (J_in1);
	\draw[label]
		node[left=2pt of oi1] {$A$}
		node[left=2pt of oi2] {$B$}
		node[right=2pt of oo1] {$C$}
		node[right=2pt of oo2] {$D$}
	;
\end{tikzpicture}
	}
\]

The comonoid structure from \cref{prop.bimonoids} tells us how the above network
will behave as a learning algorithm with respect to quadratic error. The
splitting wire will send the same parameter to both implementations $I$ and $J$,
and it will update itself based on the sum of the requests received from $I$ and
$J$.
\end{example}

\begin{remark} \label{rem.idag}
It has suited our purposes to simply consider the category $\nnet$ of neural
networks. That said, neural networks intuitively do have both monoidal and
bimonoid structure: we can place networks side by side to represent two networks
run in parallel, and we can add multiple inputs and duplicate outputs to each
node in a neural network as we like.

In fact, the category $\nnet$ can be generalised to a symmetric monoidal
category with bimonoids on each object. This generalisation is the strict
symmetric monoidal category $\idag$ of \emph{idags}---interfaced directed
acyclic graphs---which has been previously studied as an important structure in
concurrency, as well as for its elegant categorical properties \cite{FC}. 

It is also desirable that each functor $I^\sigma\colon \nnet \to \para$ implementing
neural networks as parametrised functions factors as $\nnet\to\idag\to\para$,
and indeed this can be done. Moreover, the factor $\idag \to\para$ is
symmetric monoidal and preserves bimonoid structures. 
\end{remark}

\section{Example: deep learning}\label{sec.examples}

In this section we explicitly compute an example of the functoriality of
implementing a neural network as a supervised learning algorithm. For this we
fix an activation function $\sigma$, as well as the quadratic error function, and a step size $\ve>0$. This respectively defines functors
$I^\sigma\colon \nnet \to \para$ and $L_{\ve,e}\colon \para \to \lrn$ by \cref{thm.main,prop.nnet_para}. In particular, we shall
see that $L_{\ve,e}$ implements the usual backpropagation algorithm with quadratic error
and step size $\ve$ on a neural network with activation function $\sigma$. To simplify notation, we'll write $I$ for $I^\sigma$.

Consider the following network, which has a single hidden layer:
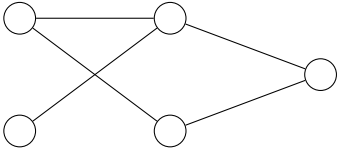
\[
\begin{tikzpicture}[xscale=2,yscale=1.5]
\node[draw, circle, minimum size=12pt, inner sep=1pt] (2) at (0,0) {};
\node[draw, circle, minimum size=12pt, inner sep=1pt] (1) at (0,1) {};
\node[draw, circle, minimum size=12pt, inner sep=1pt] (y) at (1,0) {};
\node[draw, circle, minimum size=12pt, inner sep=1pt] (x) at (1,1) {};
\node[draw, circle, minimum size=12pt, inner sep=1pt] (o) at (2,.5) {};
\draw (1) to (x);
\draw (1) to (y);
\draw (2) to (x);
\draw (x) to (o);
\draw (y) to (o);
\end{tikzpicture}
\]
Call this network $A$; it is a morphism $A\colon 2 \to 1$ in the category
$\nnet$ of neural networks. The image of $A$ under the functor $I\colon\nnet\to\para$
is the parametrised function $
I_A\colon (\rr^5\times \rr^3) \times \rr^2 \longrightarrow \rr$
defined by
\[
  I_A(p,q,a) =
\sigma\big(q_1\sigma(p_{11}a_1+p_{12}a_2+p_{1b})+q_2\sigma(p_{21}a_1+p_{2b})+q_b\big).
\]
Here the parameter space is $\rr^5 \times \rr^3$, since there is a weight for
each of the three edges in the first layer, a bias for each of the two nodes in
the intermediate column, a weight for each of the two edges in the second, and a bias
for the output node. The input space is $\rr^2$, since there are two neurons on
the leftmost side of the network, and the output space is $\rr$, since there is
a single neuron on the rightmost side. 

We write the entries of the parameter space $\rr^5 \times \rr^3$ as $p_{11}$,
$p_{12}$, $p_{21}$, $p_{1b}$, $p_{2b}$, $q_1$, $q_2$, and $q_b$, where $p_{ji}$
represents the weight on the edge from the $i$th node of the first column to the
$j$th node of the second column, $p_{jb}$ represents the bias at the $j$th node
of the second column, $q_j$ represents the weight on the edge from the $j$th
node of the second column to the unique node of the final column, and $q_b$
represents the bias at the output node.

Suppose we wish to train this network. A training method is given by the functor
$L_{\ve,e}$, which turns this parametrised function $I_A$ into a supervised learning
algorithm. In particular, given a training datum
pair $(a,c)$ in $\rr^2 \times \rr$, we wish to obtain a map $\rr^5 \times \rr^3
\to \rr^5 \times \rr^3$ that updates the value of $(p,q)$. As we have chosen to
define $L_{\ve,e}$ by using
gradient descent with respect to the quadratic error function and an $\ve$ step
size, this map is precisely the update map given by the $L_{\ve,e}$-image of $I_A$ in $\lrn$.
That is, this parametrised function maps to the learning algorithm
$(\rr^5 \times \rr^3, I_A, U_A,r_A)$, where
\[
U_A\colon (\rr^5\times \rr^3) \times \rr^2 \times \rr \longrightarrow \rr^5
\times \rr^3
\]
is defined by
\begin{align*}
U_A(p,q,a,c) &= (\begin{smallmatrix}p \\ q\end{smallmatrix})- \ve
\nabla_{p,q} \tfrac12\lVert I_A(p,q,a) - c\rVert^2 \\
& = \begin{pmatrix} p_{11} -\ve
(I_A(p,q,a)-c)\dot{\sigma}(\gamma)q_1\dot{\sigma}(\beta_1)a_1
\\
	p_{12} -\ve
(I_A(p,q,a)-c)\dot{\sigma}(\gamma)q_1\dot{\sigma}(\beta_1)a_2
\\
p_{21} -\ve
(I_A(p,q,a)-c)\dot{\sigma}(\gamma)q_2\dot{\sigma}(\beta_2)a_1
\\
p_{1b} -\ve(I_A(p,q,a)-c)\dot{\sigma}(\gamma)q_1\dot{\sigma}(\beta_1)
\\
p_{2b} -\ve(I_A(p,q,a)-c)\dot{\sigma}(\gamma)q_1\dot{\sigma}(\beta_2)
\\
q_{1} -\ve(I_A(p,q,a)-c)\dot{\sigma}(\gamma)\sigma(\beta_1)
\\
q_{2} -\ve(I_A(p,q,a)-c)\dot{\sigma}(\gamma)\sigma(\beta_2)
\\
q_b -\ve(I_A(p,q,a)-c)\dot{\sigma}(\gamma)
\end{pmatrix},
\end{align*}
and
$
r_A\colon (\rr^5\times \rr^3) \times \rr^2 \times \rr \longrightarrow \rr^2$
is defined by
\begin{align*}
r_A(p,q,&a,c) = a - \nabla_{a} \tfrac12\lVert I_A(p,q,a) - c\rVert^2\\
=&
\begin{pmatrix} a_{1} -\ve
(I_A(p,q,a)-c)\dot{\sigma}(\gamma)(q_1\dot{\sigma}(\beta_1)p_{11}+q_2\dot{\sigma}(\beta_2)p_{21})
\\
a_{2} -\ve
(I_A(p,q,a)-c)\dot{\sigma}(\gamma)q_1\dot{\sigma}(\beta_1)p_{12}
\end{pmatrix}
\end{align*}
where $\gamma$ is such that $I_A(p,q,a) = \sigma(\gamma)$, where $\beta_1 =
p_{11}a_1+p_{12}a_2+p_{1b}$, where $\beta_2 = p_{21}a_1+p_{2b}$, and where
$\dot{\sigma}$ is the derivative of the activation function $\sigma$.
(Explicitly, $\gamma
=q_1\sigma(p_{11}a_1+p_{12}a_2+p_{1b})+q_2\sigma(p_{21}a_1+p_{2b})+q_b$.)
Note that $U_A$ executes gradient descent as claimed.

The above expression for $U_A$ is complex. It, however, reuses computations like
$\gamma$, $\beta_1$, and $\beta_2$ repeatedly. To simplify computation, we might
try to factor it. A factorisation can be obtained from the neural net itself.
Note that the above net may be written as the composite of two layers. The first
layer $B\colon 2 \to 2$
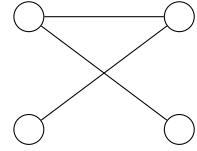
\[
\begin{tikzpicture}[xscale=2,yscale=1.5]
\node[draw, circle, inner sep=4pt] (2) at (0,0) {};
\node[draw, circle, inner sep=4pt] (1) at (0,1) {};
\node[draw, circle, inner sep=4pt] (y) at (1,0) {};
\node[draw, circle, inner sep=4pt] (x) at (1,1) {};
\draw (1) to (x);
\draw (1) to (y);
\draw (2) to (x);
\end{tikzpicture}
\]
maps to the parametrised function
\begin{align*}
I_B\colon \rr^5 \times \rr^2 &\longrightarrow \rr^2; \\
(p,a) &\longmapsto \begin{pmatrix} \sigma(p_{11}a_1+p_{12}a_2+p_{1b}) \\
\sigma(p_{21}a_1+p_{2b})\end{pmatrix}
\end{align*}
which in turn has update and request functions
\begin{align*}
U_B\colon \rr^5 \times \rr^2 \times \rr^2 &\longrightarrow \rr^5;\\
(p,a,b) &\longmapsto 
\begin{pmatrix} 
p_{11}-\ve(I_B(p,a)_1-b_1)\dot{\sigma}(\beta_1)a_1
\\
p_{12}-\ve(I_B(p,a)_1-b_1)\dot{\sigma}(\beta_1)a_2
\\
p_{21}-\ve(I_B(p,a)_2-b_2)\dot{\sigma}(\beta_2)a_1
\\
p_{1b}-\ve(I_B(p,a)_2-b_2)\dot{\sigma}(\beta_1)
\\
p_{2b}-\ve(I_B(p,a)_2-b_2)\dot{\sigma}(\beta_2)
\end{pmatrix}
\end{align*}
and
$
r_B\colon \rr^5 \times \rr^2 \times \rr^2 \longrightarrow \rr^2;
$
where
\begin{align*}
&r_B(p,a,b) \\
&=\begin{pmatrix} 
a_{1}-(I_B(p,a)_1-b_1)\dot{\sigma}(\beta_1)p_{11}+(I_B(p,a)_2-b_2)\dot{\sigma}(\beta_2)p_{21})
\\
a_{2}-(\sigma(I_B(p,a)_1-b_1)\dot{\sigma}(\beta_1)p_{12}
\end{pmatrix}
\end{align*}

The second layer $C\colon 2 \to 1$ 
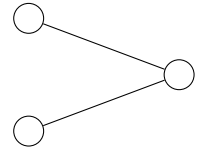
\[
\begin{tikzpicture}[xscale=2,yscale=1.5]
\node[draw, circle, inner sep=4pt] (y) at (1,0) {};
\node[draw, circle, inner sep=4pt] (x) at (1,1) {};
\node[draw, circle, inner sep=4pt] (o) at (2,.5) {};
\draw (x) to (o);
\draw (y) to (o);
\end{tikzpicture}
\]
represents the parametrised function
\begin{align*}
I_C\colon \rr^3 \times \rr^2 &\longrightarrow \rr; \\
(q,b) &\longmapsto \sigma(q_1b_1+q_2b_2+q_b).
\end{align*}
which in turn has update and request functions
\begin{align*}
U_C\colon \rr^3 \times \rr^2 \times \rr &\longrightarrow \rr^2;\\
(q,b,c) &\longmapsto 
\begin{pmatrix} 
q_1-\ve(I_C(q,b)-c)\dot{\sigma}(q_1b_1+q_2b_2+q_b)b_1
\\
q_2-\ve(I_C(q,b)-c)\dot{\sigma}(q_1b_1+q_2b_2+q_b)b_2
\\
q_b-\ve(I_C(q,b)-c)\dot{\sigma}(q_1b_1+q_2b_2+q_b)
\end{pmatrix}
\end{align*}

\begin{align*}
r_C\colon \rr^3 \times \rr^2 \times \rr &\longrightarrow \rr^2;\\
(q,b,c) &\longmapsto 
\begin{pmatrix} 
b_1-(I_C(q,b)-c)\dot{\sigma}(q_1b_1+q_2b_2+q_b)q_1
\\
b_2-(I_C(q,b)-c)\dot{\sigma}(q_1b_1+q_2b_2+q_b)q_2
\end{pmatrix}
\end{align*}
Thus the layers map respectively to the learners $(\rr^5, I_B,U_B,r_B)$ and
$(\rr^3, I_C,U_C,r_C)$. 

Functoriality says that we may recover $U_A$ and $r_A$ as composites $U_A = U_B
\ast U_C$ and $r_A = r_B \ast r_C$. For example, we can check this is true for
the first coordinate $p_{11}$:
\begin{align*}
U_B\ast U_C(p,q,a,c)_{11} &=p_{11} - \ve(I(p,a)_1 - s(q,I(p,a),c)_1)\dot{\sigma}(\beta_1)a_1 \\
&=p_{11} - \ve(J(q,I(p,a))-c)\\
&\qquad\dot{\sigma}(q_1I_1(p,a)+q_2I_2(p,a)+q_b)q_1\dot{\sigma}(\beta_1)a_1 \\
&=U_A(p,q,a,c)_{11}
\end{align*}
In particular, the functoriality describes how to factor the expressions for
the entries of $U_A$ and $r_A$ in a way that allows us to parallelise the
computation and to efficiently reuse expressions.

\section{Discussion}\label{sec.discuss}

To summarise, in this paper we have developed an algebraic framework to describe
composition of supervised learning algorithms. In order to do this, we have
identified the notion of a request function as the key distinguishing feature of
compositional learning. This request function allows us to construct training
data for all sub-parts of a composite learning algorithm from training data for
just the input and output of the composite algorithm.

This perspective allows us to carefully articulate the structure of the
backpropogation algorithm. In particular, we see that:
\begin{itemize}
\item An activation function $\sigma$ defines a functor from neural network
architectures to parametrised functions.
\item A step size $\ve$ and an error function $e$ define a
functor from parametrised functions to supervised learning algorithms. 
\item The update function for the learning algorithm defined by this functor is specified by gradient
descent.
\item The request function for the learning algorithm defined by this functor is
specified by backpropagation.
\item Bimonoid structure in the category of learning algorithms allows us to
understand neural nets, including variants such as convolutional ones, as
generated from three basic algorithms.
  \item Neural networks provide a simple, compositional language for specifying
    learning algorithms.
  \item Composition of learners, along with the
    fact that gradients are quicker to compute for lower dimensional
    spaces, expresses the speed up in learning provided by
    backpropagation.
\end{itemize}

We close with some remarks on further directions.

\subsection{More general error functions} \label{subsec.general_error}

To apply our main theorem, and hence understand backpropagation as a functor, we
require certain derivatives of our chosen error function to be invertible.
Some commonly used error functions, however, do not quite obey these
conditions. For example, \emph{cross entropy} is an error function that is similar to
quadratic error, but often leads to faster convergence. Cross entropy is given by
\[
  e(x,y) = y\ln x+(1-y)\ln(1-x).
\]
This does not supply an example of the main theorem, as the derivative is not
defined when $x =0,1$. 

It is, however, quite close to an example. There are two ways in which the
practical method differs from our theory. First, instead of using simply summing
the error to arrive at our total error $E_I$, the usual method of using cross
entropy takes the average, giving the function
\[
E_I(p,a,b) = \frac1n\sum_{j=1}^n e(I_j(p,a),b_j)
\]
where $n$ is the dimension of the codomain vector space $B$. This is quite
straightforward to model, and we show how to do this by incorporating an extra
variable $\alpha$ in our generalisation of the main theorem in Appendix A of the extended version \cite{FST}.

The second is more subtle. When $x \ne 0,1$, cross entropy has the derivative
\[
\tfrac{\partial e}{\partial x}(z,y) = \frac{y-z}{z(1-z)}.
\]
This is invertible for all $z \ne 0,1$. In practice, we consider (i) training
data $(a,b)$ such that $0 \le a_i,b_j \le 1$ for all $i,j$, as well as (ii)
$I(p,a)$ such that this implies $0< I_k(p,a)< 1$ for all $k$, assuming we start
with a suitable initial parameter $p$ and small enough step size $\ve$. In this
case $\tfrac{\partial e}{\partial x}(z,-)$ is invertible at all relevant points,
and so we can define request functions.

Indeed, in this case the request function is
\[
r_I(p,a,b)_i = a_i - \tfrac{\lvert A\rvert}{\lvert B\rvert} a_i(1-a_i) \sum_j
\frac{I_j(p,a) -b_j}{I_j(p,a)(1-I_j(p,a))}\tfrac{\partial I_j}{\partial
a_i}(p,a),
\]
while the update function is the standard update rule for gradient descent with
respect to the cross entropy.
\[
U_I(p,a,b)_k = p_k - \ve \sum_j\frac{I_j(p,a) -b_j}{I_j(p,a)(1-I_j(p,a))}
\tfrac{\partial I_j}{\partial p_k}(p,a).
\]

There is work to be done in generalising the main theorem to
accommodate error functions such as cross entropy that fail to have derivatives
at isolated points. Regardless, note that while in this case it is not
straightforward to state backpropagation as a functor from $\para$, our
analysis nevertheless still sheds light on the compositional nature of the
learning algorithm.

\subsection{Generalised networked learning algorithms} 
The category $\lrn$ contains many more morphisms than those in the images of
$\para$ under the gradient descent/backpropagation functors $L_{\ve,e}$. Indeed,
$\lrn$ does not require us to define our update and request functions using
derivatives at all. This shows that we can introduce much more general elements
than the usual neural nets into machine learning algorithms, and still use a
modular, backpropagation-like method to learn.

What might more general learning algorithms look like? As the input/output
spaces need not be Euclidean, we could choose parts of our algorithm to learn
functions that are constrained to obey certain symmetries, such as periodicity,
or equivalently being defined on a torus. Indeed, learning over
manifolds equipped with some differentiable structure is an active field of
study \cite{BBLSV}.  We might also learn nonlinear functions like rotations, or
find a way to parametrise over network architectures. 

There is a clear advantage of using gradient descent: it gives a
heuristic argument that the learning algorithm updates towards
reducing some function, which we might interpret as the error. This helps
guide the construction of a neural net. Note, however, that the category $\lrn$
sees none of this structure; it lies in the functors $L_{\ve,e}$. Thus
$\lrn$ lets us construct learning algorithms that vary the notion of error 
across the network.

Finally, neural networks are useful because they provide a simple, combinatorial
language for specifying supervised learning algorithms. In \cref{sec.bimonoids},
we saw that this fact can be cast in categorical terms as follows: neural
networks are useful as they are the language generated, using the grammar of
symmetric monoidal categories, from just a few learners (scalar multiplication,
bias, activation functor, monoid multiplication, and comultiplication). Choosing
other primitives could provide a new, similarly simple language for specifying
learning algorithms tailored to a chosen application.

\subsection{A bicategory of learners}
At present, approaches to tuning \emph{hyperparameters} of a neural network are rather
ad hoc. One such hyperparameter is the architecture of the network itself. How many
layers does the optimal neural net for a given problem have, and how many nodes
should be in each layer?

A bicategory is a generalisation of a category in which there also exist
two-dimensional morphisms connecting the usual morphisms. Learners naturally form
a bicategory. Indeed, our definition of equivalence of learners implicitly uses this structure; equivalence is just isomorphism for the following notion of 2-morphism.

\begin{definition}
  A 2-morphism $f\colon (P,I,U,r) \to
(Q,J,V,s)$ of learners is a function $f\colon P \to Q$ such that 
$J(f(p),a)= I(p,a)$, $V(f(p),a,b) = f(U(p,a,b))$, and $s(f(p),a,b) =
r(p,a,b)$.
\end{definition}

Similarly, $\para$ and $\idag$ are also naturally bicategories. Working in this bicategorical setting gives language for relating different
parametrised functions and neural network architectures. Such higher morphisms
can encode ideas such as structured expansion of networks, by adding additional
neurons or layers. 

\subsection{Learners, lenses, and open games}

We defined the category of learners to model the exchange of information between individual learning units, and how this creates a larger, composite learner. Similarly, category theory has been
used to abstractly model bidirectional programming languages and databases,
using various notions of \emph{lens}, and interacting microeconomic
games, resulting in the notion of an \emph{open game}.

These categorical analyses reveal striking structural similarities
between these three subjects, unified through the idea that at core, they study
how agents exchange and respond to information. Indeed, asymmetric lenses are
simply learners with trivial state spaces, and learners themselves are open
games obeying a certain singleton best response condition. Writing
$\mathsf{Lens}$ and $\mathsf{Game}$ for the respective categories (defined in
\cite{JR} and \cite{GHWZ}), this gives embeddings
\[ 
  \mathsf{Lens} \hookrightarrow \lrn \hookrightarrow \mathsf{Game}.
\]
Via these functors, the implementation function corresponds to the get function
of a lens and the play function of an open game, while the request function
corresponds to the put and coplay functions. The update function of the learner
corresponds to the strategy update function, known as the best response
function, for the open game \cite{FJ,Hed}. Moreover, the category $\lrn$ also embeds into a
certain category of symmetric lenses \cite{FJ}.

Lenses come with various notions of `well behavedness',
which place compatibility conditions between put and get functions. So far, in
the case of learners, we have placed no requirements that an algorithm converge
towards a function $f$ when given enough training pairs $(a,f(a))$. Examining
the lens-learner relationship may shed insight onto how not only to define
structures that learn, but that learn well.

\clearpage 

\newpage

\appendix
\subsection{Proof of \cref{prop.learners}}\label{app.proofcat}
\begin{proof}[Proof of \cref{prop.learners}]
This follows from routine checking of the axioms; we say a few words about each
case. Note that the arguments below are independent of choice of
representive of equivalence class of learner.

\paragraph{Identities}
The identity axioms are easily checked. For example, to check identity on
the left we see that $(P,I,U,r) \ast (\rr^0,\id,!,\pi_2)$ is given by $P\times \rr^0
\cong P$,\: $I(p,\id(a)) = I(p,a)$,\: $(U\ast !)(p,a,b) = U(p,a,\pi_2(I(p,a),b)) =
U(p,a,b)$,\: and $(r\ast \pi_2)(p,a,b) = r(p,a,\pi_2(I(p,a),b)) = r(p,a,b)$.

\paragraph{Associativity}
The associativity axiom is what requires that our morphisms in $\lrn$
be \emph{equivalence classes} of learners, and not simply learners themselves:
composition of learners is not associative on the nose. Indeed, this is because
products of sets are not associative on the nose: we only have isomorphisms $(P\times Q)
\times N \cong P \times (Q \times N)$ of sets, not equality. Acknowledging this,
associativity is straightforward to prove. 

Let $(P,I,U,r)\colon A \to B$,\: $(Q,J,V,s)\colon B \to C$, and $(N,K,W,t)\colon
C \to D$ be learners. The most involved item to check is the associativity of
the paired update--request function.  Computation shows
\begin{align*}
&\quad (U \ast V) \ast W \\
&= \Big(U\big(p,a,s(q,I(p,a),\gamma)\big),
V\big(q,I(p,a),\gamma\big),W\big(n,J(q,I(p,a)),d\big)\Big) \\
&= U \ast (V \ast W)
\end{align*}
where $\gamma = t\big(n,J(q,I(p,a)),d\big)$.

This equality is easier to parse using string diagrams. The composite
$(U\ast V) \ast W$ is given by the diagram
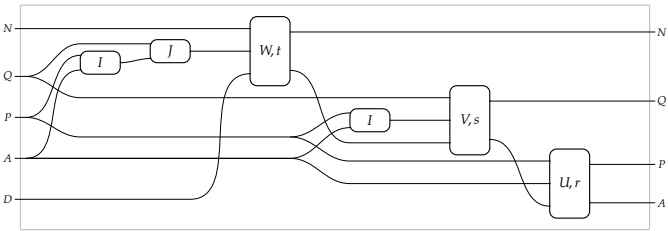
\[
\resizebox{.5\textwidth}{!}{
\begin{tikzpicture}[oriented WD]
	\node[bb port sep=1, bb={2}{1}]                            (I)     {$I$};
	\node[bb port sep=1, bb={2}{1}, above right=-1 and .5 of I](J)     {$J$};
	\node[bb port sep=2, bb={3}{2}, above right=-4 and 1 of J] (W)     {$W,t$};
	\node[bb port sep=1, bb={2}{1}, below right=2 and 1 of W](I2)     {$I$};
	\node[bb port sep=2, bb={3}{2}, above right=-4 and 1 of I2] (V)     {$V,s$};
	\node[bb port sep=2, bb={3}{2}, below right=-0.5 and 1 of V] (U)   {$U,r$};
	\node[ibb={5}{4}, fit=(I) (J) (W) (V) (U)]                          (outer) {};
	\node at ($(outer_in1'|-W_in1)-(0.3,0)$) {\footnotesize $N$};
	\node at ($(outer_in2')-(0.3,0)$) {\footnotesize $Q$};
	\node at ($(outer_in3')-(0.3,0)$) {\footnotesize $P$};
	\node at ($(outer_in4')-(0.3,0)$) {\footnotesize $A$};
	\node at ($(outer_in5')-(0.3,0)$) {\footnotesize $D$};
	\node at ($(outer_out1'|-W_out1)+(0.3,0)$) {\footnotesize $N$};
	\node at ($(outer_out2'|-V_out1)+(0.3,0)$) {\footnotesize $Q$};
	\node at ($(outer_out3'|-U_out1)+(0.3,0)$) {\footnotesize $P$};
	\node at ($(outer_out4'|-U_out2)+(0.3,0)$) {\footnotesize $A$};
	\draw (outer_in1|-W_in1) to (W_in1);
	\draw let \p1=(I.south west), \p2=($(J_in1)$), \n1=\bbportlen in
		(outer_in2) -- (outer_in2') to (\x1+\n1, \y2) -- (J_in1);
	\draw let \p1=(I.south west), \p2=($(V_in1)$), \n1=\bbportlen in
		(outer_in2') to (\x1-\n1, \y2) -- (V_in1);
	\draw (outer_in3) -- ($(\bbportlen,0)+(outer_in3')$) to (I_in1);
	\draw let \p1=(W.south east), \p2=($(outer_in3')$), \p3=(I2_in1),
	        \p4=(U_in1), \p5=(I.south west), \n1=\bbportlen in
		(\x2+\n1,\y2) to (\x5,.5*\y3+.5*\y4) -- (\x1-\n1, .5*\y3+.5*\y4) to (I2_in1);
	\draw let \p1=(W.south east), \p2=($(outer_in3')$), \p3=(I2_in1), \p4=(U_in1), \n1=\bbportlen in
		(\x1-\n1, .5*\y3+.5*\y4) to (\x3-\n1, \y4) -- (U_in1);
	\draw (outer_in4') -- ($(\bbportlen,0)+(outer_in4')$) to (I_in2);
	\draw let \p1=(W.south east), \p2=($(outer_in4')$), \n1=\bbportlen in
		(outer_in4) -- (\x1-\n1, \y2) to (I2_in2);
	\draw let \p1=(W.south east), \p2=($(outer_in4')$), \p3=(I2.south
	west), \p4=(U_in2), \n1=\bbportlen in
		(outer_in4') -- (\x1-\n1, \y2) to (\x3-\n1, \y4) -- (U_in2);
	\draw let \p1=(J.south east), \p2=($(outer_in5)$), \n1=\bbportlen in
		(outer_in5) -- (\x1+\n1, \y2) to (W_in3);
	\draw (I_out1) to (J_in2);
	\draw (J_out1) to (W_in2);
	\draw (W_out1) to (outer_out1|-W_out1);
	\draw let \p1=(I2.north west), \p2=($(V_in3)$), \n1=\bbportlen in
		(W_out2) to (\x1-\n1, \y2) -- (V_in3);
	\draw (I2_out1) to (V_in2);
	\draw (V_out1) to (outer_out2|-V_out1);
	\draw (V_out2) to (U_in3);
	\draw (U_out1) to (outer_out3|-U_out1);
	\draw (U_out2) to (outer_out4|-U_out2);
\end{tikzpicture}
}
\]
while the composite $U\ast (V\ast W)$ is given by
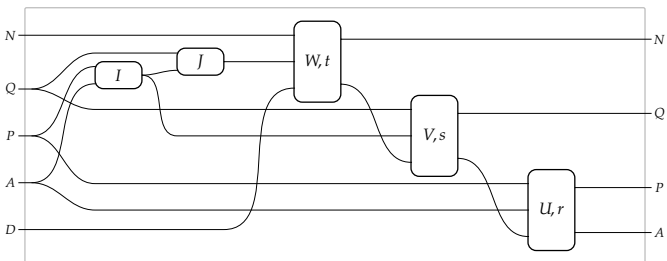
\[
\resizebox{.5\textwidth}{!}{
\begin{tikzpicture}[oriented WD]
	\node[bb port sep=1, bb={2}{1}]                            (I)     {$I$};
	\node[bb port sep=1, bb={2}{1}, above right=-1 and .5 of I](J)     {$J$};
	\node[bb port sep=2, bb={3}{2}, above right=-4 and 1 of J] (W)     {$W,t$};
	\node[bb port sep=2, bb={3}{2}, below right=-0.5 and 1 of W] (V)     {$V,s$};
	\node[bb port sep=2, bb={3}{2}, below right=-0.5 and 1 of V] (U)   {$U,r$};
	\node[ibb={5}{4}, fit=(I) (J) (W) (V) (U)]                          (outer) {};
	\node at ($(outer_in1'|-W_in1)-(0.3,0)$) {\footnotesize $N$};
	\node at ($(outer_in2')-(0.3,0)$) {\footnotesize $Q$};
	\node at ($(outer_in3')-(0.3,0)$) {\footnotesize $P$};
	\node at ($(outer_in4')-(0.3,0)$) {\footnotesize $A$};
	\node at ($(outer_in5')-(0.3,0)$) {\footnotesize $D$};
	\node at ($(outer_out1'|-W_out1)+(0.3,0)$) {\footnotesize $N$};
	\node at ($(outer_out2'|-V_out1)+(0.3,0)$) {\footnotesize $Q$};
	\node at ($(outer_out3'|-U_out1)+(0.3,0)$) {\footnotesize $P$};
	\node at ($(outer_out4'|-U_out2)+(0.3,0)$) {\footnotesize $A$};
	\draw (outer_in1|-W_in1) to (W_in1);
	\draw let \p1=(I.south west), \p2=($(J_in1)$), \n1=\bbportlen in
		(outer_in2) -- (outer_in2') to (\x1+\n1, \y2) -- (J_in1);
	\draw let \p1=(I.south west), \p2=($(V_in1)$), \n1=\bbportlen in
		(outer_in2') to (\x1-\n1, \y2) -- (V_in1);
	\draw (outer_in3) -- (outer_in3') to (I_in1);
	\draw let \p1=(I.south west), \p2=($(U_in1)$), \n1=\bbportlen in
		(outer_in3') to (\x1-\n1, \y2) -- (U_in1);
	\draw (outer_in4) -- (outer_in4') to (I_in2);
	\draw let \p1=(I.south west), \p2=($(U_in2)$), \n1=\bbportlen in
		(outer_in4') to (\x1-\n1, \y2) -- (U_in2);
	\draw let \p1=(J.south east), \p2=($(outer_in5)$), \n1=\bbportlen in
		(outer_in5) -- (\x1+\n1, \y2) to (W_in3);
	\draw (I_out1) to (J_in2);
	\draw let \p1=(J.south west), \p2=($(V_in2)$), \n1=\bbportlen in
		(I_out1) to (\x1-\n1, \y2) -- (V_in2);
	\draw (J_out1) to (W_in2);
	\draw (W_out1) to (outer_out1|-W_out1);
	\draw (W_out2) to (V_in3);
	\draw (V_out1) to (outer_out2|-V_out1);
	\draw (V_out2) to (U_in3);
	\draw (U_out1) to (outer_out3|-U_out1);
	\draw (U_out2) to (outer_out4|-U_out2);
\end{tikzpicture}
}
\]
To prove these two diagrams represent the same function, observe that the
function $(I(p,a),I(p,a))\colon P \times A \to B \times B$ can be drawn in the
following two ways:
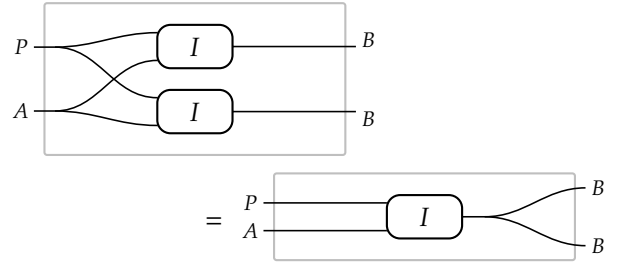
\begin{multline*}
\begin{aligned}
\begin{tikzpicture}[oriented WD]
	\node[bb port sep=1, bb={2}{1}]                            (I)     {$I$};
	\node[bb port sep=1, bb={2}{1}, below= of I]              (I2)     {$I$};
	\node[ibb={2}{2}, fit=(I) (I2)]                          (outer) {};
	\node at ($(outer_in1')-(0.3,0)$) {\footnotesize $P$};
	\node at ($(outer_in2')-(0.3,0)$) {\footnotesize $A$};
	\node at ($(outer_out1')+(0.3,0)$|-I_out1) {\footnotesize $B$};
	\node at ($(outer_out2')+(0.3,0)$|-I2_out1) {\footnotesize $B$};
	\draw (outer_in1) -- (outer_in1') to (I_in1);
	\draw (outer_in2) -- (outer_in2') to (I_in2);
	\draw (outer_in1') to (I2_in1);
	\draw (outer_in2') to (I2_in2);
	\draw (I_out1') to (outer_out1|-I_out1);
	\draw (I2_out1') to (outer_out2|-I2_out1);
\end{tikzpicture}
\end{aligned}
\\
=
\begin{aligned}
\begin{tikzpicture}[oriented WD]
	\node[bb port sep=1, bb={2}{1}]                            (I)     {$I$};
	\node[ibb={2}{2}, fit=(I)]                          (outer) {};
	\node at ($(outer_in1'|-I_in1)-(0.3,0)$) {\footnotesize $P$};
	\node at ($(outer_in2'|-I_in2)-(0.3,0)$) {\footnotesize $A$};
	\node at ($(outer_out1')+(0.3,0)$) {\footnotesize $B$};
	\node at ($(outer_out2')+(0.3,0)$) {\footnotesize $B$};
	\draw (outer_in1|-I_in1) -- (I_in1);
	\draw (outer_in2|-I_in2) -- (I_in2);
	\draw (I_out1) -- ($(I_out1)+(.2,0)$) to (outer_out1);
	\draw ($(I_out1)+(.2,0)$) to (outer_out2);
\end{tikzpicture}
\end{aligned}
\end{multline*}
This equality, and the associativity of the diagonal map, implies the equality
of the previous two diagrams, and hence the associativity of the update and
request composites.

\paragraph{Monoidality} 
It is straightforward to check the above is a monoidal
product, with unit given by the one-element set $\{*\}$.

Indeed, note that we have now shown that $\lrn$ is a category.  There exists a
functor from the category $\set$ of sets and functions to $\lrn$.  This functor
maps each set to itself, and each function $f \colon A \to B$ to the trivially
parametrised function $\overline f \colon \rr^0 \times A \to B$.  Note that
$(\set, \times)$ is a monoidal category, and let $\alpha$, $\rho$, and $\lambda$
respectively denote the associator, right unitor, and left unitor for
$(\set,\times)$. The images of these maps under this trivial parametrisation
functor $\overline{(\cdot)}$, written $\overline \alpha$, $\overline \rho$, and
$\overline \lambda$, are the corresponding structure maps for $(\lrn, \,\prl\,)$
as a symmetric monoidal category. 

The naturality of these maps, as well as the axioms of a symmetric monoidal
category, then follow in a straightforward way from the corresponding facts in
$\set$.
\bigskip

Thus we have defined a symmetric monoidal category.
\end{proof}

\subsection{Proof of \cref{thm.main}}\label{sec.proof_main}
We prove a slightly more general theorem, incorporating an extra variable
$\alpha$ so as to better describe the case of cross entropy as discussed in 
\cref{subsec.general_error}.
\begin{theorem}[Generalisation of \cref{thm.main}]
Fix $\ve>0$, $\alpha\colon \nn \to \rr_{>0}$, and $e(x,y)\colon \rr \times \rr \to
\rr$ differentiable such that $\frac{\partial e}{\partial x}(z,-)\colon \rr \to
\rr$ is invertible for each $z \in \rr$. 

Then we can define a faithful, injective-on-objects, symmetric monoidal functor
\[
L\colon\para \longrightarrow \lrn
\]
that sends each parametrised function $I\colon P \times A \to B$ to the learner
$(P,I,U_I,r_I)$ defined by
\[
  U_I(p,a,b) \coloneqq p - \varepsilon\nabla_pE_I(p,a,b)
\]
and
\[
  r_I(p,a,b) \coloneqq f_a\bigg(\frac{1}{\alpha_B}\nabla_a E_I(p,a,b)\bigg),
\]
where $f_a$ is component-wise application of the inverse to $\frac{\partial
e}{\partial x}(a_i,-)$ for each $i$, and
\[
E_I(p,a,b) \coloneqq \alpha_B \sum_j e(I_j(p,a),b_j).
\]
\end{theorem}
\begin{proof}

The functor $L$ is by definition injective-on-objects. Since $I$ maps to $(P,I,U_I,r_I)$,
the functor $L$ is injective on morphisms, and hence will give a faithful functor.

Let $I\colon P \times \rr^n \to \rr^m$ and $J\colon Q \times \rr^m \to \rr^\ell$
be parametrised functions. We show that the composite of their images is equal
to the image of their composite.

\paragraph{Update functions}
By definition the composite of the update functions of $I$ and $J$ is given by
\begin{align*}
&\quad (U_I\ast U_J)(p,q,a,c) \\
&= \Big(U_I(p,a,r_J(q,I(p,a),c)),\; U_J(q,I(p,a),c)\Big)\\
&= \Big(p - \ve\nabla_p E_I(p,a,r_J(q,I(p,a),c)),\; q-\ve\nabla_q
E_J(q,I(p,a),c)\big),
\end{align*}
while the update function of the composite $I \ast J$ is
\begin{align*}
U_{I \ast J}(p,q,a,c) = \Big(p - \ve\nabla_p E_{I \ast J}(p,q,a,c),\; q -
\ve\nabla_qE_{I \ast J}(p,q,a,c)\Big).
\end{align*}
To show that these are equal, we thus must show that the following equations hold
\begin{align}\label{eqn:p_entry}
	\nabla_p E_I(p,a,r_J(q,I(p,a),c)) &= \nabla_p E_{I \ast J}(p,q,a,c)\\\label{eqn:q_entry}
	\nabla_q E_J(q,I(p,a),c) &= \nabla_qE_{I \ast J}(p,q,a,c).
\end{align}

We first consider Equation \ref{eqn:p_entry}:
\begin{align*}
  & \quad \nabla_p E_I(p,a,r_J(q,I(p,a),c))
\\&= 
  \nabla_p \alpha_B \sum_i e\left(I_i(p,a),r_J(q,I(p,a),c)_i\right)
\tag{def $E_I$}\\&= 
  \bigg(\alpha_B \sum_i \frac{\partial e}{\partial
  x}\left(I_i(p,a),r_J(q,I(p,a),c)_i\right)\frac{\partial I_i}{\partial
  p_\ell}(p,a)\bigg)_\ell
\tag{def $\nabla_p$}\\&=
	\bigg(\alpha_B \sum_i \frac{\partial e}{\partial
  x}\bigg(I_i(p,a),f_{I(p,a)}\Big(\frac{1}{\alpha_B}\nabla_bE_J(q,I(p,a),c)\Big)_i\bigg)\frac{\partial
  I_i}{\partial p_\ell}(p,a)\bigg)_\ell
\tag{def $r_J$}\\&= 
  \bigg(\alpha_B \sum_i \frac1{\alpha_B}\left(\nabla_bE_J(q,I(p,a),c)\right)_i\frac{\partial
  I_i}{\partial p_\ell}(p,a)\bigg)_\ell
\tag{def $f$}\\&= 
  \bigg(\alpha_B \sum_i \frac1{\alpha_B}\big(\nabla_b
  \alpha_C \sum_j e(J_j(q,I(p,a)),c_j)\big)_i
  \frac{\partial
  I_i}{\partial p_\ell}(p,a)\bigg)_\ell
\tag{def $E_J$}\\&= 
  \bigg(\sum_i \alpha_C \sum_j\frac{\partial e}{\partial
  x}(J_j(q,I(p,a)),c_j)\frac{\partial J_j}{\partial b_i}(q,I(p,a))\frac{\partial
  I_i}{\partial p_\ell}(p,a)\bigg)_\ell
\tag{def $\nabla_b$}\\&= 
  \bigg(\alpha_C \sum_j\frac{\partial e}{\partial
  x}(J_j(q,I(p,a)),c_j)\frac{\partial J_j}{\partial p_\ell}(q,I(p,a))\bigg)_\ell
\tag{chain rule}\\&= 
  \nabla_p\alpha_C \sum_j e(J_j(q,I(p,a)),c_j)
\tag{def $\nabla_p$}\\&= 
  \nabla_p E_{I \ast J}(p,q,a,c)
\tag{def $E_{J\ast I}$}
\end{align*}
note the shift to coordinate-wise reasoning, that $f$ is defined as the inverse
to $\partial e$, and the use of the chain rule. Here, $i$, $j$, and $\ell$ are indexing over the dimensions of $B$, $C$, and $P$ respectively.

Equation \ref{eqn:q_entry} simply follows from the definition of the error
function; we need not even take derivatives:
\[
E_J(q,I(p,a),c) = \alpha \sum_i e\big(J(q,I(p,a))_i,c_i\big) =
E_{I \ast J}(p,q,a,c)_m 
\]
Thus we have shown $U_I\ast U_J = U_{I \ast J}$, as desired.

\paragraph{Request functions}
We must prove that the following equation holds:
\[
r_I(p,a,r_J(q,I(p,a),c)) = r_{I \ast J}(p,q,a,c)
\]
This follows due to the chain rule, in the exact same manner as for updating
$p$, but swapping the roles of $a$ and $p$ in the proof of Equation \ref{eqn:p_entry}:
\begin{align*}
\qquad r_I(p,a,r_J(q,I(p,a),c))
&= f_a\bigg(\frac{1}{\alpha_B}\nabla_a E_I(p,a,r_J(q,I(p,a),c))\bigg) \\
&= f_a\bigg(\frac{1}{\alpha_B}\nabla_a E_{I \ast J}(p,q,a,c)\bigg) \\
&= r_{I \ast J}(p,q,a,c)
\end{align*}

\paragraph{Identities} The identity on the object $A$ in the category of
parametrised functions is the projection $\id_A\colon \rr^0 \times A \to A$. The image of
$\id_A$ has trivial update function, since the parameter space is trivial. The
request function is given by
\[
r_{\id_A}(0, a, b) = f_a(\tfrac1{\alpha_B} \nabla_a (\alpha_B\sum_ie(a_i,b_i)))
= \bigg(f_a(\frac{\partial e}{\partial a_i}(a_i,b_i))\bigg)_i = b.
\]
This is exactly the identity map $(\rr^0,\id_A,!,\pi_2)$ in $\lrn$.

\paragraph{Monoidal structure} The functor $L$ is a monoidal functor. That is, the
learner given by the monoidal product of parametrised functions is equal to the
monoidal product of the learners given by those same functions, up to the
standard isomorphisms $\rr^n \times \rr^m \cong \rr^{n+m}$. To see that this
is true, suppose we have parametrised functions $I\colon P \times A \to B$ and
$J\colon Q \times C \to D$. Their tensor is $I \prl J \colon (P \times Q)
\times (A \times C) \to B \times D$. Note that $E_{I\; \prl\; J}(p,q,a,c,b,d)
= E_I(p,a,b) + E_J(q,c,d)$. Thus the update function of their tensor is given by
\begin{align*}
&\quad U_{I\; \prl\; J}(p,q,a,c,b,d) \\
&= \big(p - \ve\nabla_pE_{I\; \prl\; J}(p,q,a,c,b,d)),\; q - \ve\nabla_qE_{I\; \prl\; J}(p,q,a,c,b,d)\big)\\
&= \big(p - \ve\nabla_pE_{I}(p,a,b)),\; q - \ve\nabla_qE_{J}(q,c,d)\big)\\
&= \big(U_I(p,a,b),\; U_J(q,c,d)\big)
\end{align*}
and similarly the request function is
\begin{align*}
&\quad r_{I \;\prl\; J}(p,q,a,c,b,d) \\
&= f_{(a,c)}\big(\tfrac1{\alpha_{B \times D}}\nabla_{(a,c)}E_{I\;\prl\;
J}(p,q,a,c,b,d)\big) \\
&= f_{(a,c)}\bigg(\tfrac1{\alpha_{B \times
D}}\nabla_{(a,c)}\alpha_{B\times D} \Big(\sum_i e(I(p,a)_i,b_i)+\sum_j
e(J(q,c)_j,d_j)\Big)\bigg) \\
&= \Big(f_{a}\big(\tfrac1{\alpha_B}\nabla_a E_I(p,a,b)\big),\; 
f_{c}\big(\tfrac1{\alpha_D}\nabla_c E_J(q,c,d)\big)\Big) \\
&= \Big(r_I(p,a,b),\; r_J(q,c,d)\Big)
\end{align*}
Thus image of the tensor is the tensor of the image.
\end{proof}

\subsection{Background on category theory}\label{app.cats}

\subsubsection{Symmetric monoidal categories}
A symmetric monoidal category is a setting for composition for
network-style diagrammatic languages like neural networks. A prop is a
particularly simple sort of strict symmetric monoidal category.

First, let us define a category. We specify a \define{category}
$\mathsf C$  using the
data: 
\begin{itemize}
\item a collection $X$ whose elements are called \emph{objects}.
\item for every pair $(A,B)$ of objects, a set $[A,B]$ whose elements are called
\emph{morphisms}.
\item for every triple $(A,B,C)$ of objects, a function $[A,B] \times [B,C] \to
[A,C]$ call the \emph{a composition rule}, and where we write $(f,g) \mapsto
f;g$.
\end{itemize}
This data is subject to the axioms
\begin{itemize}
\item identity: for all objects $A$ there exists $\id_A \in [A,A]$ such that for
all $f\in [A,B]$ and $g \in [B,A]$ we have $\id_A;f= f$ and $g;\id_A =g$.
\item associativity: for all $f \in [A,B]$, $g \in [B,C]$ and $h \in [C,D]$ we
have $(f;g);h = f;(g;h)$.
\end{itemize}

The main object of our interest, however, is a particular type of category,
known as a symmetric monoidal category. For a \define{symmetric
monoidal category} $\mathsf C$, we further require the data:
\begin{itemize}
\item for every pair $(A,B)$ of objects, another object $A \otimes B$ in $X$.
\item for every quadruple $(A,B,C,D)$ of objects a function $[A,B] \times [C,D]
\to [A\otimes C, B \otimes D]$ called the \emph{monoidal product}.
\end{itemize}

Using this data, we may draw networks. We think of the objects as being various
types of wire, and a morphism $f$ in $[A_1\otimes \dots \otimes A_n,B_1 \otimes
\dots \otimes B_m]$ as a box with wires of types $A_i$ on the
left and wires of types $B_i$ on the right. Here are some pictures.
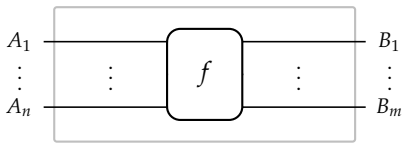
\[
\begin{tikzpicture}[oriented WD]
	\node[symbb port sep=1.4, symbb={3}{3}]                            (I)     {$f$};
	\node[ibb={3}{3}, fit=(I)]                          (outer) {};
	\node at ($(outer_in1'|-I_in1)-(0.4,0)$) {\footnotesize $A_1$};
	\node at ($(outer_in2'|-I_in2)-(0.4,-0.2)$) {\footnotesize $\vdots$};
	\node at ($(outer_in3'|-I_in3)-(0.4,0)$) {\footnotesize $A_n$};
	\node at ($(outer_out1'|-I_out1)+(0.4,0)$) {\footnotesize $B_1$};
	\node at ($(outer_out2'|-I_out2)+(0.4,0.2)$) {\footnotesize $\vdots$};
	\node at ($(outer_out3'|-I_out3)+(0.4,0)$) {\footnotesize $B_m$};
	\node at ($(I_in2)-(0.5,-0.2)$) {\footnotesize $\vdots$};
	\node at ($(I_out2)+(0.5,0.2)$) {\footnotesize $\vdots$};
	\draw (outer_in1|-I_in1) to (I_in1);
	\draw (outer_in3|-I_in3) to (I_in3);
	\draw (I_out1') to (outer_out1|-I_out1);
	\draw (I_out3') to (outer_out3|-I_out3);
\end{tikzpicture}
\]

By connecting wires of the same type, we can draw more complicated pictures. For
example:
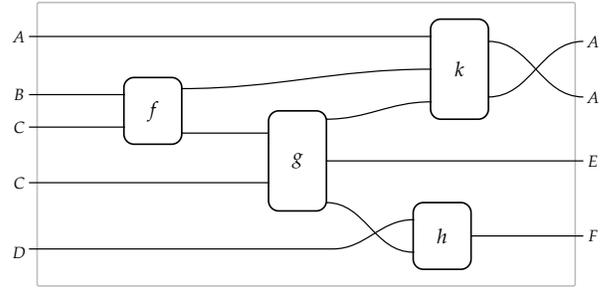
\[
  \resizebox{.45\textwidth}{!}{
\begin{tikzpicture}[oriented WD]
	\node[symbb port sep=1, bb={2}{2}]                           (F)    {$f$};
	\node[symbb port sep=2, bb={2}{3}, below right=-2 and 1 of F](G)   {$g$};
	\node[bb port sep=2, bb={2}{1}, below right=-0.5 and 1 of G](H)	{$h$};
	\node[bb port sep=2, bb={3}{2}, above right=-0.5 and 1.2 of G](K){$k$};
	\node[ibb={5}{4}, fit=(F) (K) (G) (H)]                          (outer) {};
	\node at ($(outer_in1'|-K_in1)-(0.3,0)$) {\footnotesize $A$};
	\node at ($(outer_in2'|-F_in1)-(0.3,0)$) {\footnotesize $B$};
	\node at ($(outer_in3'|-F_in2)-(0.3,0)$) {\footnotesize $C$};
	\node at ($(outer_in4'|-G_in2)-(0.3,0)$) {\footnotesize $C$};
	\node at ($(outer_in5'|-H_in2)-(0.3,0)$) {\footnotesize $D$};
	\node at ($(outer_out1'|-K_out2)+(0.3,0)$) {\footnotesize $A$};
	\node at ($(outer_out2'|-K_out1)+(0.3,0)$) {\footnotesize $A$};
	\node at ($(outer_out3'|-G_out2)+(0.3,0)$) {\footnotesize $E$};
	\node at ($(outer_out4'|-H_out1)+(0.3,0)$) {\footnotesize $F$};
	\draw (outer_in1|-K_in1) to (K_in1);
	\draw (outer_in2|-F_in1) to (F_in1);
	\draw (outer_in3|-F_in2) to (F_in2);
	\draw (outer_in4|-G_in2) to (G_in2);
	\draw (outer_in5) -- ($(outer_in5)+(3.5,0)$) to (H_in1);
	\draw (F_out1) to (K_in2);
	\draw (F_out2) to (G_in1|-F_out2);
	\draw (G_out1) to (K_in3);
	\draw (G_out3) to (H_in2);
	\draw (G_out2) to (outer_out3|-G_out2);
	\draw (H_out1) to (outer_out4|-H_out1);
	\draw (K_out1) to (outer_out2|-K_out2);
	\draw (K_out2) to (outer_out1|-K_out1);
\end{tikzpicture}
}
\]
The key point of a network is that any such picture must have an unambiguous
interpretation as a morphism. The use of string diagrams to represent morphisms
in a monoidal category is formalised in \cite{JS}.

To form what is known as a \define{strict} symmetric monoidal category, the above data must obey
additional axioms that ensure it captures the above intuition of behaving like a
network. These axioms are
\begin{itemize}
\item interchange: for all $f \in [A,B]$, $g \in [B,C]$, $h \in [D,E]$, $k \in
[E,F]$ we have $(f;g)\otimes (h;k) = (f \otimes h);(g \otimes k)$.
\item monoidal identity: there exists an object $I$ such that $I \otimes A = A =
A \otimes I$.
\item monoidal associativity: for all objects $A,B,C$ we have $(A \otimes
B)\otimes C= A \otimes (B\otimes C)$.
\item symmetry: for all pairs of objects $A,B$ we have morphisms
$\sigma_{A,B} \in [A\otimes B,B \otimes A]$ such
that $\sigma_{A,B}; \sigma_{B,A} = \id_{A \otimes B}$, and that for all $f \in
[A,C]$, $g \in [B,D]$ we have $(f\otimes g);\sigma_{C \otimes D} = \sigma_{A
\otimes B};(f\otimes g)$. 
\end{itemize}

More generally, symmetric monoidal categories require these axioms only to be
true up to natural isomorphism. More detail can be found in \cite{Mac}.

\begin{example}
An example of a symmetric monoidal category is $(\set, \times)$, where
our objects are a set of each cardinality, and morphisms are functions between
them.  The monoidal product is given by the cartesian product of sets.
\end{example}
\begin{example}
Another example of a symmetric monoidal category is $(\fvect,\oplus)$,
where our objects are finite-dimensional vector spaces, morphisms are linear
maps, and the monoidal product is given by the direct sum of vector spaces.
\end{example}

\subsubsection{Functors}
A functor is a way of reinterpreting one category in another, preserving the
algebraic structure. In other words, a functor is the notion of
structure preserving map for categories, in analogy with linear transformations
as the structure preserving maps for vector spaces, and group homomorphisms as
the structure preserving maps for groups.

Formally, given categories $\mathsf C, \mathsf D$, a \define{functor} $F\colon
\mathsf{C} \to \mathsf{D}$ sends every object $A$ of $\mathsf C$ to an object
$FA$ of $\mathsf D$, every morphism $f\colon A \to B$ in $\mathsf C$ to a
morphism $Ff\colon FA \to FB$ in $\mathsf D$, such that $F1=1$ and $Ff;Fg =
F(f;g)$. 

A functor between symmetric monoidal categories is a \define{symmetric
monoidal functor} if $FI_{\mathsf C}=I_{\mathsf D}$, where $I$ is the monoidal
unit for the relevant category, and if there exist
isomorphisms $F(A \otimes B) \cong FA \otimes FB$ natural in objects $A,B$ of
$\mathsf C$. We say that the functor is a \define{strict} symmetric monoidal
functor if these isomorphisms are in fact equalities.

We also say that a functor is \define{faithful} if $Ff=Fg$ only when $f=g$, and
\define{injective-on-objects} if the map from objects of $\mathsf C$ to objects
of $\mathsf D$ is injective. 

\subsubsection{Bimonoids}
A \define{bimonoid} in a symmetric monoidal category is an object $A$ together
with morphisms that obey certain axioms. These morphisms have names and types:
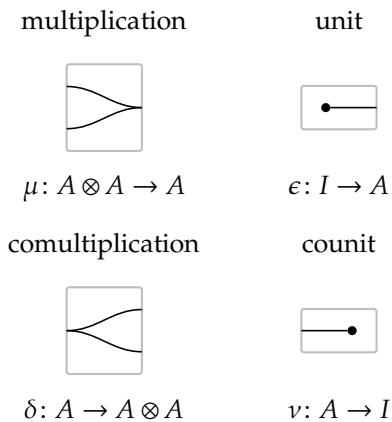
\[
  \begin{tikzcd}[row sep=1pt]
  \textrm{multiplication}&
  \textrm{unit}
  \\\\
\begin{tikzpicture}[oriented WD]
	\node[blankbb={2}{1}]                          (outer) {};
	\draw (outer_in1) to (outer_out1);
	\draw (outer_in2) to (outer_out1);
\end{tikzpicture}
&
\begin{tikzpicture}[oriented WD]
	\node[blankbb={0}{1}]                          (outer) {};
	\node[draw, fill, circle, inner sep=1pt] (u) at ($(outer_out1)-(.45,0)$) {};
	\draw (u) to (outer_out1);
\end{tikzpicture}
\\
\mu\colon A\otimes A \to A &
\epsilon\colon I \to A
\\\\\\\\
  \textrm{comultiplication}&
  \textrm{counit} 
  \\\\
\begin{tikzpicture}[oriented WD]
	\node[blankbb={1}{2}]                          (outer) {};
	\draw (outer_in1) to (outer_out1);
	\draw (outer_in1) to (outer_out2);
\end{tikzpicture}
&
\begin{tikzpicture}[oriented WD]
	\node[blankbb={1}{0}]                          (outer) {};
	\node[draw, fill, circle, inner sep=1pt] (u) at ($(outer_in1)+(.45,0)$) {};
	\draw (outer_in1) to (u);
\end{tikzpicture}
\\
\delta\colon A \to A \otimes A &
\nu\colon A \to I
\end{tikzcd}
\]
Note that these diagrams are informal, but useful, special depictions of
these morphisms. More formally, for example, the diagram 
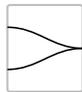
\[
\begin{tikzpicture}[oriented WD]
	\node[blankbb={2}{1}]                          (outer) {};
	\draw (outer_in1) to (outer_out1);
	\draw (outer_in2) to (outer_out1);
\end{tikzpicture}
\]
for the multiplication $\mu$ is a shorthand for the string diagram
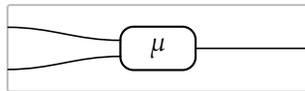
\[
\begin{tikzpicture}[oriented WD, bb port sep=1pt]
	\node[bb={2}{1}]                          (m) {$\mu$};
	\node[blankbb={2}{1}, fit = (m)]                          (outer) {};
	\draw (outer_in1) to (m_in1);
	\draw (outer_in2) to (m_in2);
	\draw (m_out1) to (outer_out1);
\end{tikzpicture}
\]

These morphisms must obey the axioms:
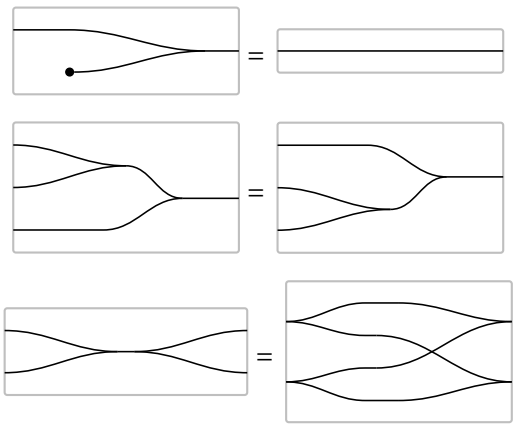
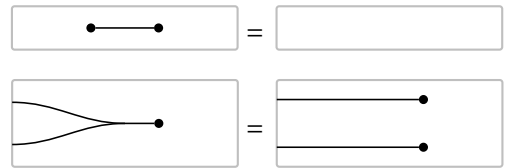
\[
\begin{aligned}
\begin{tikzpicture}[oriented WD]
	\coordinate (u) {};
	\node[blankbb={2}{1}, fit=(u)]                          (outer) {};
	\node[draw, fill, circle, inner sep=1pt] (u) at ($(outer_in2)+(.5,0)$) {};
	\draw (outer_in1) -- ($(outer_in1)+(.5,0)$) to ($(outer_out1)-(.3,0)$)
	-- (outer_out1);
	\draw (u) to ($(outer_out1)-(.3,0)$);
\end{tikzpicture}
\end{aligned}
=
\begin{aligned}
\begin{tikzpicture}[oriented WD]
	\coordinate (u) {};
	\node[blankbb={1}{1}, fit=(u)]                          (outer) {};
	\draw (outer_in1) to (outer_out1);
\end{tikzpicture}
\end{aligned}
\]
\[
\begin{aligned}
\begin{tikzpicture}[oriented WD]
	\coordinate (u) {};
	\coordinate (l) at ($(u)+(.5,-1.5)$) {};
	\coordinate (cent) at ($(u)+(0,-2)$) {};
	\node[blankbb={3}{1}, fit= (u) (cent)]                          (outer) {};
	\draw (outer_in1) to (u);
	\draw (outer_in2) to (u);
	\draw (outer_in3) -- ($(outer_in3)+(.8,0)$)to (l);
	\draw (u) to (l);
	\draw (l) to (outer_out1|-l);
\end{tikzpicture}
\end{aligned}
=
\begin{aligned}
\begin{tikzpicture}[oriented WD]
	\coordinate (l) {};
	\coordinate (u) at ($(l)+(.5,1.5)$) {};
	\coordinate (cent) at ($(l)+(0,2)$) {};
	\node[blankbb={3}{1}, fit= (cent) (l)]                          (outer) {};
	\draw (outer_in1) -- ($(outer_in1)+(.8,0)$) to (u);
	\draw (outer_in2) to (l);
	\draw (outer_in3) to (l);
	\draw (l) to (u);
	\draw (u) to (outer_out1|-u);
\end{tikzpicture}
\end{aligned}
\]
\[
\begin{aligned}
\begin{tikzpicture}[oriented WD]
	\coordinate (l) {};
	\coordinate (r) at ($(l)+(.15,0)$) {};
	\node[blankbb={2}{2}, fit= (r) (l)]                          (outer) {};
	\draw (outer_in1) to (l);
	\draw (outer_in2) to (l);
	\draw (l) to (r);
	\draw (r) to (outer_out1);
	\draw (r) to (outer_out2);
\end{tikzpicture}
\end{aligned}
=
\begin{aligned}
\begin{tikzpicture}[oriented WD]
	\coordinate (a) {};
	\coordinate (b) at ($(a)-(0,1.5)$) {};
	\coordinate (c) at ($(b)-(0,1.5)$) {};
	\coordinate (d) at ($(c)-(0,1.5)$) {};
	\node[blankbb={2}{2}, fit= (a) (b) (c) (d)]     (outer) {};
	\draw (outer_in1) to ($(a)-(.3,0)$)-- (a);
	\draw (outer_in1) to ($(b)-(.3,0)$)--($(b)-(.2,0)$);
	\draw (outer_in2) to ($(c)-(.3,0)$)--($(c)-(.2,0)$);
	\draw (outer_in2) to ($(d)-(.3,0)$)--(d);
	\draw (a) to (outer_out1);
	\draw ($(c)-(.2,0)$) to (outer_out1);
	\draw ($(b)-(.2,0)$) to (outer_out2);
	\draw (d) to (outer_out2);
\end{tikzpicture}
\end{aligned}
\]
\[
\begin{aligned}
\begin{tikzpicture}[oriented WD]
	\coordinate (c) at (0,0) {};
	\node[blankbb={1}{1}, fit=(c)]                          (outer) {};
	\node[draw, fill, circle, inner sep=1pt] (l) at ($(outer_in1)+(.7,0)$) {};
	\node[draw, fill, circle, inner sep=1pt] (r) at ($(outer_out1)-(.7,0)$) {};
	\draw (l) to (r);
\end{tikzpicture}
\end{aligned}
=
\begin{aligned}
\begin{tikzpicture}[oriented WD]
	\coordinate (c) at (0,0) {};
	\node[blankbb={1}{1}, fit=(c)]                          (outer) {};
\end{tikzpicture}
\end{aligned}
\]
\[
\begin{aligned}
\begin{tikzpicture}[oriented WD]
	\node[draw, fill, circle, inner sep=1pt] (c) {};
	\coordinate (cent) at ($(c)-(.3,0)$) {};
	\node[blankbb={2}{1}, fit=(cent)]                    (outer) {};
	\draw (outer_in1) to (cent) -- (c);
	\draw (outer_in2) to (cent);
\end{tikzpicture}
\end{aligned}
=
\begin{aligned}
\begin{tikzpicture}[oriented WD]
	\node[draw, fill, circle, inner sep=1pt] (c) {};
	\node[draw, fill, circle, inner sep=1pt] (d) at ($(c)-(0,2.2)$) {};
	\coordinate (cent) at ($(c)-(.3,1.1)$) {};
	\node[blankbb={2}{1}, fit=(cent)]                    (outer) {};
	\draw (outer_in1|-c) to (c);
	\draw (outer_in2|-d) to (d);
\end{tikzpicture}
\end{aligned}
\]
and their mirror images.

Note that the second axiom above, called associativity, implies that all maps with codomain
$1$ constructed using only products of the multiplication and the identity map
are equal. It is thus convenient, and does not cause confusion, to define the
following notation:
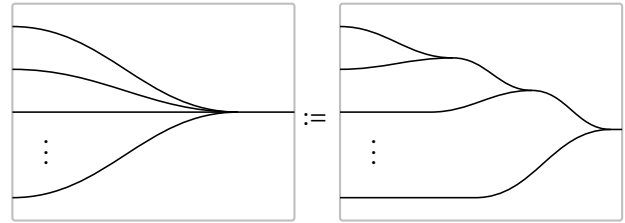
\[
\begin{aligned}
\begin{tikzpicture}[oriented WD]
	\coordinate (u) {};
	\coordinate (cent) at ($(u)+(.5,-5)$) {};
	\node[blankbb={5}{1}, fit= (u) (cent)]                          (outer) {};
	\coordinate (o) at ($(outer_out1)-(.5,0)$) {};
	\node at ($(outer_in4)+(.3,.5)$) {$\vdots$};
	\draw (outer_in1) to (o);
	\draw (outer_in2) to (o);
	\draw (outer_in3) to (o);
	\draw (outer_in5) to (o);
	\draw (o) to (outer_out1);
\end{tikzpicture}
\end{aligned}
:=
\begin{aligned}
\begin{tikzpicture}[oriented WD]
	\coordinate (u) {};
	\coordinate (l) at ($(u)+(.7,-1.5)$) {};
	\coordinate (m) at ($(l)+(.7,-1.8)$) {};
	\coordinate (cent) at ($(u)+(.5,-5)$) {};
	\node[blankbb={5}{1}, fit= (u) (cent)]                          (outer) {};
	\node at ($(outer_in4)+(.3,.5)$) {$\vdots$};
	\draw (outer_in1) to (u);
	\draw (outer_in2) to (u);
	\draw (u) to (l);
	\draw (outer_in3) -- ($(outer_in3)+(.8,0)$)to (l);
	\draw (l) to (m);
	\draw ($(outer_in5)$) -- ($(outer_in5)+(1.2,0)$)to (m);
	\draw (m) to (outer_out1|-m);
\end{tikzpicture}
\end{aligned}
\]
We define the mirror image notation for iterated comultiplications.

These morphisms, and the axioms they obey, allow network diagrams to be drawn.
First, the morphisms $\mu$, $\epsilon$, $\delta$, and $\nu$ respectively give
interpretations to pairwise merging, initializing, splitting in pairs, and deleting
edges. The associativity and coassociativity axioms, for example, then give
unique interpretation to $n$-ary merging and $n$-ary splitting, as described
above.

\begin{example}
Each object in $\fvect$ can be equipped with the structure of a
bimonoid. Indeed, given a vector spaces $V$, the multiplication $\mu\colon V
\oplus V \to V$ takes a pair $(u,v)$ to $u+v$, the unit $\epsilon\colon 0 \to V$
maps the unique element $0$ of the $0$-dimensional vector space to the zero
vector in $V$, the comultiplication $\delta\colon V \to V \oplus V$ maps $v$ to
$(v,v)$, and the counit $\nu\colon V \to 0$ maps every vector $v \in V$ to zero.
It is standard linear algebra to check that these maps obey the bimonoid axioms;
see \cite{BE,BSZ} for details.
\end{example}
\end{document}